\documentclass[11pt]{article}
\usepackage[colorlinks=true, 
urlcolor=blue,linkcolor=blue,citecolor=blue]{hyperref}
\usepackage{amsmath,amsfonts,amssymb,amsthm}
\usepackage{mathrsfs}
\usepackage{color}
\usepackage{amscd}
\usepackage{enumerate}

\newtheorem{theorem}{Theorem}[section]
\newtheorem{lemma}[theorem]{Lemma}

\newtheorem{proposition}[theorem]{Proposition}

\theoremstyle{remark}

\numberwithin{equation}{section}

\def\R{{\mathbb R}}

\def\Pb{{\mathbb{P}}}			
\def\Var{{\mathrm{Var}}}		
\def\E{{\mathbb{E}}}			
\def\F{ {\mathcal{F}} }			

\def\eps{\varepsilon}
\def\L{ {\mathcal{L}} }

\def\U{ {\mathcal{U}} }
\def\lft{\mathscr{L}}
\def\astar{ {a^*} }
\def\Ct{ {\mathcal{C}} }
\newcommand{\I}{\mathcal{I}}

\title{Corrector theory for MsFEM and HMM in random 
media}
\author{Guillaume Bal \and Wenjia Jing\thanks{Department of Applied 
Physics and Applied
Mathematics, Columbia University, NY 10027.  Email addresses: {\tt gb2030@columbia.edu}, {\tt wj2136@columbia.edu}.
}}
\begin{document}
\maketitle

\begin{abstract}
We analyze the random fluctuations of several multi-scale algorithms such as the multi-scale finite element method (MsFEM) and the finite element heterogeneous multi-scale method (HMM), that have been developed to solve partial differential equations with highly heterogeneous coefficients. Such multi-scale algorithms are often shown to correctly capture the homogenization limit when the highly oscillatory random medium is stationary and ergodic. This paper is concerned with the random fluctuations of the solution about the deterministic homogenization limit. We consider the simplified setting of the one dimensional elliptic equation, where the theory of random fluctuations is well understood. We develop a fluctuation theory for the multi-scale algorithms in the presence of random environments with short-range and long-range correlations. 

For a given mesh size $h$, we show that the fluctuations converge in distribution in the space of continuous paths to Gaussian processes as the correlation length $\eps\to0$. We next derive the limit of such Gaussian processes as $h\to0$ and compare this limit with the distribution of the random fluctuations of the continuous model. When such limits agree, we conclude that the multi-scale algorithm captures the random fluctuations accurately and passes the corrector test. This property serves as an interesting Benchmark to assess the behavior of the multi-scale algorithm in practical situations where the assumptions necessary for the theory of homogenization are not met. 

What we find is that the computationally more expensive methods MsFEM, and HMM with a choice of parameter $\delta=h$, correctly capture the random fluctuations both for short-range and long-range oscillations in the medium. The less expensive method HMM with $\delta<h$ correctly captures the fluctuations for long-range oscillations and strongly amplifies their size in media with short-range oscillations. We present a modified scheme with an  intermediate computational cost that captures the random fluctuations in all cases.  

\smallskip



\end{abstract}

\section{Introduction}
\label{sec:intro}

Differential equations with highly oscillatory coefficients arise naturally in many areas of applied sciences and engineering from the analysis of composite materials to the modeling of geological basins. Often, it is impossible to solve the microscopic equations exactly, in which case we can either solve macroscopic models when they exist or devise multi-scale algorithms that aim to capture as much of the microscopic scale as possible. 
When the coefficients are periodic functions or stationary and ergodic processes, the solution to the heterogeneous equation is often well approximated  by the deterministic solution to a homogenized equation. This is the well known \emph{homogenization theory}. In many applications, such as parameter estimation and uncertainty quantification, estimating the random fluctuations (finding the random corrector) in the solution is as important as finding its homogenized limit. 

Finding the homogenized coefficients may be a daunting computational task and the assumptions necessary to the applicability of homogenization theory may not be met \cite{JKO-SV-94,Kozlov,PV-81}. Several numerical methodologies have been developed to find accurate approximations of the solution without solving all the details of the micro-structure \cite{A-SINUM-04,EMZ-JAMS-05,EHW-SINUM-00}. Examples include the multi-scale finite element method (MsFEM) and the finite element heterogeneous multi-scale method (HMM). Such schemes are shown to perform well in the homogenization regime, in the sense that they approximate the solution to the homogenized equation without explicitly calculating any macroscopic, effective medium, coefficient. Homogenization theory thus serves as a Benchmark that ensures that the multi-scale scheme performs well in controlled environments, with the hope that it will still perform well in non-controlled environments, for instance when ergodicity and stationarity assumptions are not valid. 

This paper aims to present another Benchmark for such multi-scale numerical schemes that addresses the limiting stochasticity of the solutions. We calculate the limiting (probability) distribution of the random corrector given by the multi-scale algorithm when the correlation length of the medium tends to $0$ at a fixed value of the discretization size $h$. We then compare this distribution to the distribution of the corrector of the continuous equation. When these distributions are close, in the sense that the $h-$dependent distribution converges to the continuous distribution as $h\to0$, we deduce that the multi-scale algorithm asymptotically correctly captures the randomness in the solution and passes the random corrector test.

The above proposal requires a controlled environment in which the theory of correctors is available. There are very few equations for which this is the case \cite{B-CLH-08,BGMP-AA-08,FOP-SIAP-82}. In this paper, we initiate such an analysis in the simple case of the following one-dimensional second-order elliptic equation: 
\begin{equation}
\left\{
\begin{aligned}
& - \frac{d}{dx} a(\frac{x}{\eps},\omega) \frac{d}{dx} u_\eps (x, \omega) = 
f(x), \, x \in (0, 1),\\
& u_\eps(0, \omega) = u_\eps (1, \omega) = 0.
\end{aligned}
\right.
\label{eq:rode}
\end{equation}
Here, the diffusion coefficient $a(\frac{x}{\eps},\omega)$ is modeled as a random process, where $\omega$ denotes the realization in an abstract probability space $(\Omega,\F,\Pb)$ in which the random process and all limits considered here are constructed. The correlation length $\eps$ is much smaller  than $L=1$, the length of the domain, which makes the random coefficient highly oscillatory.

We chose this equation for two reasons. First, many multi-scale numerical schemes have been developed to solve its generalization in higher dimensions. Second, both the homogenization and the corrector theory for this elliptic equation are well known. When $a(x,\omega)$ is stationary, ergodic, and uniformly elliptic, i.e., $\lambda \le a(x,\omega) \le \Lambda$ for almost every $x$ and $\omega$, then the solution $u_\eps$ converges to the following homogenized equation with deterministic and constant coefficient (this generalizes to higher dimensions as well \cite{Kozlov, PV-81}):
\begin{equation}
\left\{
\begin{aligned}
& - \frac{d}{dx} \astar \frac{d}{dx} u_0 (x) = f(x), \, x \in (0, 
1),\\
& u_0 (0) = u_0 (1) = 0.
\end{aligned}
\right.
\label{eq:hode}
\end{equation}
The coefficient $\astar$ is the harmonic mean of $a(x,\omega)$. 
Furthermore, if the deviation of $1/a(x,\omega)$ from its mean $1/\astar$ 
is strongly mixing with integrable mixing coefficient, a notion that will 
be defined in the next section and which in particular implies that the 
random coefficient has short range correlation (SRC), then the corrector 
theory in \cite{BP-99} shows that \begin{equation}
\frac{u_\eps - u_0}{\sqrt{\eps}}(x) \xrightarrow[\eps \to 
0]{\mathrm{distribution}} \U(x; W),
\end{equation}
where $\U(x; W)$ is a Gaussian random process that may be conveniently 
written as a Wiener integral; see \eqref{eq:cltone} below. The convergence 
above is in the sense of distribution in the space of continuous paths, 
which will be  denoted by $\Ct([0,1])$. When the deviation $1/a - 
1/\astar$ does not de-correlate fast enough, the normalization factor 
$\sqrt{\eps}$ is no longer correct. In fact, for a large class of random 
processes with long range correlation (LRC) defined more precisely later, 
it is proved in \cite{BGMP-AA-08} that a similar convergence result holds 
with the following  modifications: The normalization factor should be 
replaced by $\eps^{\frac \alpha 2}$, where $0< \alpha <1$ describes the 
rate of de-correlation.  The limiting process is then $\U_H(x; W^H)$, which 
is defined in \eqref{eq:clttwo} and has the form of a stochastic integral 
with respect to fractional Brownian motion. 

Let us assume that we want to solve \eqref{eq:rode} numerically. 
We denote by $h$ the discretization size and by $u^h_\eps(x)$ the solution of the 
scheme. The standard finite element solution of the deterministic limit \eqref{eq:hode} is denoted by $u^h_0(x)$. Our goal is to characterize the limiting distribution of $u^h_\eps - u^h_0$ as a random process after proper rescaling by $\eps^{\frac{\alpha \wedge 1}{2}}$. Here and below, we use $c \wedge d$ to denote $\min\{c, d\}$. Obviously, this normalization parameter is chosen to be consistent with the aforementioned corrector theories.  

We say that a numerical procedure is consistent with the corrector theory and that it passes the corrector test when the following diagram commutes:
\begin{equation}
\begin{CD}
   {\displaystyle \frac{u^h_\eps - u^h_0}{\eps^{\frac{\alpha \wedge 
1}{2}}}(x,\omega)}
   @>{h \to 0}>(i)> {\displaystyle \frac{u_\eps - 
   u_0}{\eps^{\frac{\alpha \wedge 1}{2}}}(x,\omega)}\\
   @V {\eps \to 0}V(ii)V  @V(iii)V{\eps \to 0}V\\
   {\U^h_{\alpha \wedge 1} (x; W^{\alpha \wedge 1})} @>{h \to 0}>(iv)> 
   {\U_{\alpha \wedge 1} (x; W^{\alpha \wedge 1})}.
 \end{CD}
 \label{eq:cd}
\end{equation}
Here, $W^1 = W$ is the standard Brownian motion whereas $W^\alpha = W^H$ is 
the fractional Brownian motion with Hurst index $H = 1- {\frac \alpha 2}$. 
Similarly, $\U_{\alpha \wedge 1}$ is defined so that $\U_1 = \U(x; W)$ and 
$\U_\alpha = \U_H(x; W^H)$. We use this notation to include both random 
processes with SRC and those with LRC. In the above diagram, there are four 
convergence paths to be understood in the sense of distribution in $\Ct 
([0, 1])$. The convergence path (iii) is the corrector theories in 
\cite{BP-99, BGMP-AA-08}.  In (i), $h$ is sent to zero while $\eps$ is 
fixed. To check (i) is a numerical analysis question without multi-scale 
issues since the $\eps$-scale details are resolved. Convergence in (i) can 
be obtained path-wise and not only in distribution (path (iv) may also be 
considered path-wise). The main new mathematical difficulties we address in 
this paper, therefore lie in analyzing the paths (ii) and (iv). 

Our main results are stated for the multi-scale finite element method (MsFEM) \cite{HW-JCP-97, HWC-MC-99, H-MFEM-03}, and the finite element heterogeneous multi-scale method (HMM) \cite{EE-HMM-REV, EMZ-JAMS-05}. The main idea of MsFEM is to use multi-scale basis functions constructed from the local solutions of the elliptic operator in \eqref{eq:rode} on domains of size $h\ll1$. The main idea of HMM is to construct these basis functions on possibly smaller patches. HMM thus involves heterogeneous computations on domains of size $\delta\leq h$.

The analysis of MsFEM in  random settings is done in \cite{CS-SIAMNA-08, E-PHD}. HMM in random media was considered in \cite{EMZ-JAMS-05}. These works show that both MsFEM and HMM pass the test of homogenization theory in the sense that as $\eps\to0$, the limiting solution is a consistent discretization (i.e., with error converging to $0$ as $h\to0$) of the homogenized limit. This paper analyzes the corrector theory of both multi-scale methods.

For MsFEM and HMM with the parameter choice $\delta=h$ (these two methods are then the same in dimension $n=1$), the above diagram commutes. Path (i) 
holds as a result of standard finite element analysis. Path (ii) holds due 
to the self-averaging effect of integrals with oscillatory integrand 
$a^{-1}(x/\eps) - \astar^{-1}$, and $\U^h_{\alpha \wedge 1}$ is explicitly 
characterized as a Gaussian process. The expression for this process 
is complicated but converges to the right corrector as $h\to0$.

For HMM with $\delta<h$, path (i) does not hold because there is no homogenization effect when $h \ll \eps$.  Passing the corrector test then means that  (ii) followed by  (iv) yields the same result as (iii). Path (ii) is obtained with  $\U^h_{\alpha \wedge 1}$ a Gaussian process. HMM with $\delta<h$ then performs differently for SRC and LRC. In the first case, the limit process as $h\to0$ in (iv) is $(h/\delta)^{\frac{1}{2}} \U$: the limit is an amplified version of the theoretical corrector when $\delta<h$. In that sense, HMM with $\delta<h$ does not pass the corrector test. In the LRC case, however, and somewhat surprisingly at first, HMM passes the corrector test for all $0<\delta\leq h$. As a consequence, HMM with $\delta<h$ finds the corrector up to a rescaling coefficient that depends on the structure of the random medium and whose estimation may prove difficult for  environments in which the corrector theory is not valid or not known.

We consider a modification of the MSFEM and HMM schemes that is computationally more expensive than HMM with $\delta\ll h$ but  passes the corrector test independent of the underlying random structure. The proposed corrector test and our analysis, in spite of the limited scope of the one-dimensional equation, thus provide some guidance on the accuracy one may expect from a multi-scale algorithm based on a given computational cost and a given level of prior information about the heterogeneous medium (such as, e.g., its correlation properties) one is willing to make.

The rest of the paper is organized as follows. In the next section, we 
present our models for the random medium and formulate the main 
results of the paper. The derivations rely on establishing explicit expressions for 
$u^h_\eps - u^h_0$ . In section \ref{sec:stra}, we derive such expressions 
for a general numerical scheme satisfying specific assumptions on the structure of the associated 
stiffness matrix.  We show how to prove (ii) using these formulas. 
In  section \ref{sec:wctwo}, we show how to prove (iv) in the general setting.  
Both convergences in (ii) and (iv) can be viewed as weak convergence of 
measures in $\Ct$, which can be shown as usual by obtaining the convergence 
of finite dimensional distributions and establishing tightness. In section 
\ref{sec:msfem}, we apply the general framework to MsFEM and prove that 
diagram \eqref{eq:cd} commutes in this case. In section \ref{sec:hmm}, we 
consider HMM in both media with SRC or LRC. We emphasize where the 
amplification effect comes from in the case of SRC and why this effect 
disappears in the case of long-range media.  In section \ref{sec:disc}, we 
discuss methods to  eliminate the amplification effect of HMM. Some conclusions are offered in section \ref{sec:conclusion} while technical lemmas on fourth-order moments of random processes are postponed to the appendix. 

\section{Main results on the corrector test}
\label{sec:main}

In this section, we introduce our hypotheses on the random media, describe the multi-scale algorithms MsFEM and HMM and formulate our main convergence results.

To simplify notation, we drop the dependency in $\omega$ when this does not cause confusion. 
We define $a_\eps(x)=a(x/\eps)$. For a function $g$ in $L^p(D)$, we 
denote its norm by $\|g\|_{p,D}$.  When $D$ is the unit interval we drop 
the symbol $D$.  The natural space for \eqref{eq:rode} and \eqref{eq:hode} 
is the Hilbert space $H^1_0$. By the Poincar\'{e} inequality,  the 
semi-norm of $H^1_0$  defined by $|u|_{H^1, D} = \|du/dx\|_{2,D}$,
is equivalent to the standard norm. We use the notation $C$ to denote 
constants that may vary from line to line. When $C$ depends only on the 
elliptic constants $(\lambda, \Lambda)$, we refer to it as a {\it 
universal} constant.  Finally, the Einstein summation convention is frequently used: two repeated indices, such as in $c_i d^i$, are 
summed over their (natural) domain of definition.

\subsection{Random media models}

We model $a(x,\omega)$ as a random process on some probability space
$(\Omega, \F, \Pb)$, with parameter $x \in \R$. The coefficient 
$a(\frac{x}{\eps},\omega)$ in \eqref{eq:rode} is obtained by rescaling the 
spatial variable. 

A process $a(x)$ is called \emph{stationary} if the joint distribution of 
any finite collection of random variables $\{a(x_i,\omega)\}_{1\le i \le 
n}$ is translation invariant. Let $\E$ denote the mathematical expectation 
with respect to $\Pb$. We define
\begin{equation}
  \frac{1}{\astar} : = \E\left\{\frac{1}{a(0,\omega)}\right\}, \quad 
q(x,\omega) = \frac{1}{a(x,\omega)} - \frac{1}{\astar}.
  \label{eq:aqdef}
\end{equation}
Since $a$ is stationary, $\astar$ is a constant and $q(x,\omega)$ is a 
stationary process with mean zero.

The (auto-)correlation function $R(x)$ of $q$ is given by:
\begin{equation}
  R(x) : = \E \{q(y)q(y+x)\} = \E\{q(0)q(x)\}.
  \label{eq:Rdef}
\end{equation}
We find that  $R$ is symmetric, i.e., $R(x) = R(-x)$ and is a 
function of positive type in the sense of \cite[\S XI]{RS}. By Bochner's theorem, the Fourier transform of $R$ is non-negative. Let us define
\begin{equation}
\sigma^2 = \int_{\R} R(x) dx.
\label{eq:sigmadef}
\end{equation}
Then $\sigma^2$ is in $[0,\infty]$, and we assume here that $\sigma>0$. We 
say that $q(x,\omega)$ has \emph{short range correlation} (SRC) if $R$ is 
integrable. In this case, $\sigma$ is a positive finite number. We say $q$ 
has \emph{long range correlation} (LRC) if the correlation function is not 
integrable.

The decay of $R(x)$ at infinity describes the two-point decorrelation rate of the
process $q$. We also need the notion of \emph{mixing}. Consider a Borel set $A \subset \R$. Let $\F_A$ denote the Borel sub-$\sigma$-algebra of $\F$ induced by $\{q(x); x\in A\}$, and $L^2_A$ the space of $\F_A-$measurable and square integrable random variables.  A 
random process $a(x,\omega)$ is said to be $\rho$-\emph{mixing} if there 
exists some non-negative function $\rho(r): \R_+ \to \R_+$ with $\rho(r) 
\to 0$ at infinity, such that for any Borel sets $A$ and $B$ the following 
holds.
\begin{equation}
  \sup_{\xi \in L^2_A, \eta \in L^2_B} \frac{|\E\{\xi \eta\} - 
  \E\{\xi\}\E\{\eta\}|}{\Var\{\xi\}^{\frac 1 2}\Var\{\eta\}^{\frac 1 2}} 
  \le \rho(d(A,B)).
  \label{eq:rhodef}
\end{equation}
Here $d(A,B)$ is the distance between the two sets. The function $\rho(r)$ 
is called the $\rho$-mixing coefficient, and it reflects, roughly speaking, 
how fast random processes restricted on $A$ and $B$ become independent.  
More details on the notion of mixing can be found in \cite{Doukhan}. 

Our first set of assumptions on $a(x)$ and $q(x)$ are the following. We note that the third
assumption below implies the second one.

\begin{enumerate}
\item[(S1)] The random process $a(x)$ is stationary and uniformly elliptic 
with constants $(\lambda, \Lambda)$. Clearly, $\astar$ is uniformly elliptic with the same constants.

\item[(S2)] The random process $q(x)$ has SRC and we assume $\sigma > 0$ in 
\eqref{eq:sigmadef}.

\item[(S3)] The mixing coefficient $\rho(r)$ of $q(x)$ is decreasing in 
$r$, and $\rho^{\frac 1 2}(r) \in L^1(\R_+)$.
\end{enumerate}

The above assumptions are quite general. In particular, (S3) 
implies ergodicity of $q(x)$, and (S1) plus ergodicity is the standard 
assumption for homogenization theory; (S3) is the standard assumption to 
obtain central limit theorem of oscillatory integrals with integrand 
$q_\eps(x)$ as in \eqref{eq:cltbal}.

For random media with LRC, there is no general central limit theorem. A 
special family of random media investigated in \cite{BGMP-AA-08} is based 
on the following assumptions:

\begin{enumerate}
\item[(L1)] The process $q(x)$ is constructed as
\begin{equation}
q(x, \omega) = \Phi(g(x,\omega)),
\label{eq:Phidef}
\end{equation}
where $g_x$ is a stationary Gaussian random process with mean zero and 
variance one. Further, the correlation function $R_g$ of $g_x$ has the 
following asymptotic behavior:
\begin{equation}
R_g(\tau) \sim \kappa_g \tau^{-\alpha},
\label{eq:Rg}
\end{equation}
where $\kappa_g > 0 $ is a constant and $\alpha \in (0,1)$.

\item[(L2)] The function $\Phi(x)$ satisfies $|\Phi| \le q_0$ for some 
constant $q_0$, so that the process $a(x,\eps)$, constructed by the 
relation \eqref{eq:aqdef} for some positive constant $\astar$, satisfies 
uniform ellipticity with constants $(\lambda, \Lambda)$.

\item[(L3)] The function $\Phi$ integrates to zero against the standard 
Gaussian measure:
\begin{equation}
\int_{\R} \Phi(s) e^{-\frac{s^2}{2}} ds = 0.
\label{eq:odd}
\end{equation}
\end{enumerate}
The process $q(x)$ is stationary and mean-zero. More importantly, its correlation function $R(x)$
has a similar asymptotic behavior to that in \eqref{eq:Rg} with $\kappa_g$ replaced by $\kappa$, where
\begin{equation}
\kappa: = \kappa_g \left(\frac{1}{\sqrt{2\pi}} \int_{\R} s \Phi(s) 
e^{-\frac{s^2}{2}} ds\right)^2.
\label{eq:kappa}
\end{equation}
The integral above is assumed to be non-zero. Consequently, $R(x)$ is no 
longer integrable and $q(x)$ has LRC. We note that when $\alpha > 1$, the 
process constructed above has SRC and provides an example satisfying (S2).

\subsection{The multi-scale algorithms MsFEM and HMM}

We briefly introduce MsFEM and HMM and leave some details to later sections.
Assume $f \in L^2 \subset H^{-1}$. The weak solution to \eqref{eq:rode} is the unique
function $u_\eps \in H_0^1(0,1)$ such that
\begin{equation}
A_\eps(u,v) = F(v), \quad \forall v \in H^1_0(0,1).
\label{eq:wfmsfem}
\end{equation}
The associated bilinear and linear forms are defined as
\begin{equation}
A_\eps(u, v) : = \int_0^1 a_\eps(x) \frac{d u}{dx} \cdot \frac{d v}{dx} dx, 
\quad F(v) : = \int_0^1 f v \ dx.
\label{eq:linforms}
\end{equation}
Existence and uniqueness of $u_\eps$ are guaranteed by the uniform ellipticity
of $a_\eps(x)$. MsFEM and the HMM are finite element 
methods and so the ideas are to approximate $H^1_0$ by finite dimensional 
space and, when necessary, to approximate $A_\eps$ by an auxiliary bilinear form.

We partition the unit interval into $N$ small intervals of size $h = 
1/N$. We use this partition for MsFEM, HMM, and the standard FEM. The FEM basis functions are piece-wise linear ``hat functions", 
each of them peaking at one nodal point and vanishing at all other 
nodal points. Denote these hat functions by $\{\phi^j_0\}$ and denote the 
subspace of $H^1_0$ they span by $V^h_0$. The standard FEM approximates $H^1_0$ by $V^h_0$.  The idea of MsFEM is to replace the hat basis functions by multi-scale basis 
functions $\{\phi^j_\eps\}$. For instance, $\phi^j_\eps$ is constructed 
by solving the elliptic equation \eqref{eq:rode} locally on the 
support of $\phi^j_0$; see \eqref{eq:msfemb} below. Denote the span of 
these multi-scale basis functions by $V^h_\eps$. Then, the MsFEM solution, 
denoted by $u^h_\eps$, is the unique function in $V^h_\eps$ so that 
\eqref{eq:wfmsfem} holds for all $v \in V^h_\eps$. The well-posedness of 
this weak formulation is a consequence of the uniform ellipticity of $a_\eps$.

HMM aims to reduce the cost of MsFEM by computing heterogeneous solutions on smaller domains. Solutions and test functions are still sought in $V^h_0$. The bilinear form $A_\eps$ is then approximated by
\begin{equation*}
A^{h,\delta}_\eps(u,v) : = \sum_{k=1}^N \left(\frac{1}{\delta} \int_{I^\delta_k} a_\eps(x) \frac{d(\lft u)}{dx} \cdot \frac{d(\lft v)}{dx} 
\ dx\right) h,
\end{equation*}
where $I^\delta_k$ is a small sub-interval of size $\delta < h$ of the $k$th interval in 
the above partition. The operator $\lft$ is defined below in \eqref{eq:lft}.  This linear operator optimally uses the micro-scale calculation on $I^\delta_k$ in the same way that $\phi^j_\eps$ does in MsFEM. 
HMM solution is the unique function $u^{h,\delta}_\eps$ in $V^h_0$ so that \eqref{eq:wfmsfem}
holds when $H^1_0$ is replaced by $V^h_0$ and $A_\eps$ is replaced by $A^{h,\delta}_\eps$.
The well-posedness of this formulation is less immediate and will be proved below.

Throughout this paper, for simplicity in presentation, we assume that the micro-scale solvers in MsFEM and HMM, i.e., equations \eqref{eq:msfemb} and \eqref{eq:lft}, are exact.

\subsection{Main theorems}

We state our main convergence theorems in the setting of MsFEM and HMM although they hold for more general schemes.

The first theorem analyzes MsFEM in the setting of SRC.
\begin{theorem}
  \label{thm:msfem}
  Let $u_\eps$ and $u_0$ be solutions to \eqref{eq:rode} and 
  \eqref{eq:hode}, respectively. Let $u^h_\eps$ be the solution to 
  \eqref{eq:rode} obtained by MsFEM and let $u^h_0$ be the standard finite 
  element approximation of $u_0$. Then we have:
  
  {\upshape (i)} Suppose that  $a(x)$ satisfies {\upshape (S1)} and $f$ is continuous on $[0,1]$.  
  Then
  \begin{equation}
    |u^h_\eps - u_\eps|_{H^1} \le \frac{h}{\lambda \pi} \|f\|_2, \quad 
    \|u^h_\eps - u_\eps\|_{2} \le \frac{h^2}{\lambda \pi^2} \|f\|_2.
    \label{eq:fem}
  \end{equation}
  Assume further that $q(x)$ satisfies   {\upshape (S2)}. Then
  \begin{equation}
   \sup_{x \in [0,1]} \left\lvert \E \left(u^h_\eps(x) - u^h_0(x)\right)^2 
\right \rvert \le C \frac{\eps}{h^2} \|R\|_{1,\R}(1+\|f\|_2),
    \label{eq:supL2}
  \end{equation}
  where $C$ is a universal constant and $R$ is the correlation function of $q$ as defined
  in \eqref{eq:Rdef}.

  {\upshape (ii)} Now assume further that $q(x)$ satisfies {\upshape (S3)}.  
Then,
  \begin{equation}
   \frac{u^h_\eps(x) - u^h_0(x)}{\sqrt{\eps}} \xrightarrow[\eps \to 
   0]{\mathrm{distribution}}
   \sigma \int_0^1 L^h (x,t) dW_t = : \U^h(x; W).
    \label{eq:cltone}
  \end{equation}
  The constant $\sigma$ is defined in \eqref{eq:sigmadef} and $W$ is 
  the standard Wiener process. The function $L^h(x,t)$ is explicitly given 
  in \eqref{eq:msfemL}. The convergence is in distribution in the space $\Ct$.

  {\upshape (iii)} Now let $h$ goes to zero, we have
  \begin{equation}
   \U^h(x; W) \xrightarrow[h \to 0]{\mathrm{distribution}} \U(x; W) : = 
   \sigma\int_0^1 L(x,t) dW_t.
    \label{eq:clttwo}
  \end{equation}
  The Gaussian process 
  $\U(x; W)$ was characterized in   {\upshape \cite{BP-99}}. The kernel $L(x,t)$ is defined as 
  \begin{equation}
    L(x,t) = {\bf 1}_{[0,x]}(t)\left(\int_0^1 F(s) ds - F(t) \right) + 
    x\left(F(t) - \int_0^1 F(s) ds\right) {\bf 
    1}_{[0,1]}(t).
    \label{eq:Ldef}
\end{equation}
Here and below, ${\bf 1}$ is the indicator function and $F(t)=\int_0^t f(s)ds$.
\end{theorem}

Equivalently, the theorem says the diagram in \eqref{eq:cd} commutes
when $q$ has SRC and that MsFEM passes the corrector test in this setting.

To prove (iv) of the diagram, we recast $L(x,t)$ as
 \begin{equation}
   L(x,t) = \astar \frac{\partial}{\partial y} G_0(x,t) \cdot \astar 
   \frac{\partial}{\partial x} u_0(t).
   \label{eq:Ldef2}
 \end{equation}
 Here $G_0$ is the Green's function of \eqref{eq:hode}. It has the 
 following expression:
 \begin{equation}
   G_0(x,y) = \left\{
   \begin{aligned}
     \astar^{-1}x(1-y), & &x \le y,\\
     \astar^{-1}(1-x)y, & &x > y.
   \end{aligned}
   \right.
   \label{eq:green}
 \end{equation}
 In particular,$G_0$ is Lipschitz continuous in each 
variable while the other is kept fixed.

The next theorem accounts for MsFEM in media with LRC. We recall that the 
random process $q(x)$ below is constructed as a function of a Gaussian 
process.
\begin{theorem}
  \label{thm:msfeml}
  Let $u_\eps$, $u_0$, $u^h_\eps$ and $u^h_0$ be defined as in the previous 
theorem. Let $q(x,\omega)$ and $a(x,\omega)$ be constructed as in {\upshape 
(L1)-(L3)}. Then we have

{\upshape (i)} \begin{equation}
\label{eq:supL2l}
\sup_{x \in [0,1]} \left\lvert \E \left(u^h_\eps(x) - u^h_0(x)\right)^2 
\right \rvert \le C \frac{1}{h}\left(\frac{\eps}{h}\right)^{\alpha},
\end{equation}
for some constant $C$ depending on $(\lambda, \Lambda)$, $\kappa$, $\alpha$, and $f$.

{\upshape (ii)} As $\eps$ goes to zero while $h$ is fixed, we have
\begin{equation}
\frac{u^h_\eps(x) - u^h_0(x)}{\eps^{\frac \alpha 2}} \xrightarrow[\eps \to 
0]{\mathrm{distribution}}
 \U^h_H(x; W^H) : = \sigma_H \int_0^1 L^h (x,t) dW^H_t.
\label{eq:cltonel}
\end{equation}
Here $H = 1 - {\frac \alpha 2}$, and $W^H_t$ is the standard fractional 
Brownian motion with Hurst index $H$. The constant $\sigma_H$ is defined as 
$\sqrt{\kappa/H(2H-1)}$. The function $L^h(x,t)$ is defined as in the 
previous theorem.

{\upshape (iii)} As $h$ goes to zero, we have
\begin{equation}
\U^h_H(x; W^H) \xrightarrow[h \to 0]{\mathrm{distribution}} \U_H(x; W^H) : 
= \sigma_H \int_0^1 L(x,t) dW^H_t.
\label{eq:clttwol}
\end{equation}
\end{theorem}
As before, this theorem says the diagram in \eqref{eq:cd} 
commutes in the current case. In particular, $\alpha < 1$, and the scaling 
is $\eps^{\frac \alpha 2}$. Thus MsFEM passes the corrector test for both 
SRC and LRC. The stochastic integrals in \eqref{eq:cltonel} and 
\eqref{eq:clttwol} have fractional Brownian motions as integrators. We give 
a short review of such integrals in the appendix. A good reference is 
\cite{VT-PTRF-00}.

The next theorem addresses the convergence properties of HMM.  
\begin{theorem}
  \label{thm:hmm}
  Let $u_\eps$ and $u_0$ be the solutions to \eqref{eq:rode} and 
  \eqref{eq:hode}, respectively. Let $u^{h,\delta}_\eps$ be the HMM solution and $u^h_0$ the standard finite element approximation of  $u_0$. 
  
  {\upshape (i)} Suppose that the random processes $a(x)$ and 
$q(x)$ satisfy {\upshape (S1)-(S3)}.  Then
  \begin{equation}
    \frac{u^{h,\delta}_\eps(x) - u^h_0(x)}{\sqrt{\eps}} \xrightarrow[\eps 
    \to 0]{\mathrm{distribution}}
    \U^{h,\delta}(x; W) \xrightarrow[h\to 
0]{\mathrm{distribution}}\sqrt{\frac{h}{\delta}}\ \U(x; W).
    \label{eq:clth}
  \end{equation}
  Here, $\U^{h,\delta}(x; W)$ is as in \eqref{eq:cltone} with $L^h$ 
replaced by $L^{h,\delta}(x,t)$, which is defined in \eqref{eq:hmmL} below.  
The process $\U(x;W)$ is as in \eqref{eq:clttwo}.

  {\upshape (ii)} Suppose instead that the random processes $a(x)$ 
and $q(x)$ satisfy {\upshape (L1)-(L3)}. Then
  \begin{equation}
    \frac{u^{h,\delta}_\eps(x) - u^h_0(x)}{\eps^{\frac \alpha 2}} 
    \xrightarrow[\eps \to 0]{\mathrm{distribution}}
    \U^{h,\delta}_H(x; W^H) \xrightarrow[h\to 
0]{\mathrm{distribution}}\U_H(x; W^H).
    \label{eq:clthl}
  \end{equation}
  Here, $\U^{h,\delta}_H(x; W^H)$ is as in \eqref{eq:cltonel} with $L^h$ 
replaced by $L^{h,\delta}$, and $\U_H(x;W^H)$ is as in \eqref{eq:clttwol}.
\end{theorem}

HMM is computationally less expensive than MsFEM when $\delta$ is much smaller 
than $h$. However, the theorem implies that this advantage comes at a 
price: when the random process $q(x)$ has SRC, HMM with $\delta<h$ amplifies the variance 
of the corrector. We will discuss methods to eliminate this effect in section 
\ref{sec:disc}. In the case of LRC, however, HMM does 
pass the corrector test even when $\delta\ll h$.

Intuitively, averaging occurs at the small scale $\delta\ll h$ for SRC.  
Since HMM performs calculations on a small fraction of each interval $h$, 
each integral needs to be rescaled by $h/\delta$ to capture the correct 
mean, which over-amplifies the size of the fluctuations. In the case of 
long-range correlations, the self-similar structure of the limiting process 
shows that the convergence to the Gaussian process occurs simultaneously at 
all scales (larger than $\eps$) and hence at the macroscopic scale. HMM may 
then be seen as a collocation method (with grid size $h$), which does 
capture the main features of the random integrals.

The amplification of the random fluctuations might be rescaled to provide the correct answer. The main difficulty is that the rescaling factor depends on the structure of the random medium and thus requires prior information or additional estimations about the medium. For general random media with no clear scale separation or no stationarity assumptions, the definition of such a rescaling coefficient might be difficult. In section \ref{sec:disc}, we present a hybrid method between HMM and MsFEM that is computationally less expensive than the MsFEM presented above while still passing the corrector test.

\section{Expression for the corrector and convergence as $\eps\to0$}
\label{sec:stra}

The starting point to prove the main theorems is to derive a formula for the corrector $u^h_\eps - u^h_0$. This is done by inverting the discrete systems yielded from the multi-scale schemes. The goal is to write the corrector as a stochastic integral linear in the random coefficient, plus error terms that are negligible in the limit; see Proposition \ref{prop:msfemL} below.
\subsection{General finite element based multi-scale schemes}

Almost all finite element based multi-scale schemes for \eqref{eq:rode} have the same main
premise: in the weak formulation \eqref{eq:wfmsfem}, we approximate $H^1_0$ by a finite dimensional space and if necessary, also approximate the bilinear form.

To describe the choices of the finite spaces, we choose a uniform partition of the unit interval into $N$ sub-intervals with size $h=1/N$. Let $x_k$ denote the $k$th grid point, with $x_0 = 0$ and $x_N = 1$, and $I_k$ the interval $(x_{k-1},x_k)$. To simplify notation in the general setting, we still denote by $V^h_\eps$ the finite space and by $\{\phi^j_\eps\}_{j=1}^{N-1}$ the basis functions. They do not necessarily coincide with those in MsFEM.

In a general scheme, the bilinear form in \eqref{eq:wfmsfem} might be modified. Nevertheless, to
simplify notations, we still denote it as $A_\eps$. The solution obtained from the scheme is then $u^h_\eps \in V^h_\eps$ such that
\begin{equation}
A_\eps(u^h_\eps, v) = F(v), \quad \forall v \in V^h_\eps.
\end{equation}
Since $V^h_\eps$ is finite dimensional, the above condition amounts to a linear
system, which is obtained by putting $u^h_\eps = U^\eps_j\phi^j_\eps$, and by
requiring the above equation to hold for all basis functions. The linear system is:
\begin{equation}
A^h_\eps U^\eps = F^\eps.
\label{eq:msfeml}
\end{equation}
Here, the vector $U^\eps$ is a vector in $\R^{N-1}$, and it has entries $U^\eps_j$. The load vector $F^\eps$ is also in $\R^{N-1}$ and has entries $F(\phi^j_\eps)$. The stiffness matrix $A^h_\eps$ is an $N-1$ by $N-1$ matrix, and its entries are $A_\eps(\phi^i_\eps, \phi^j_\eps)$. Our main assumptions on the basis functions and the stiffness matrix are the following.

\begin{enumerate}
\item[(N1)] For any $1\le j \le N-1$, the basis function $\phi^j_\eps$
is supported on $I_j\cup I_{j+1}$, and it takes the value $\delta^i_j$ at 
nodal points $\{x_i\}$. Here $\delta^i_j$ is the Kronecker symbol.

\item[(N2)] The matrix $A^h_\eps$ is symmetric and tri-diagonal.  In 
addition, we assume that there exists a vector $b_\eps \in \R^N$ with 
entries
$\{b^k_\eps\}_{k=1}^N$, so that $A^h_{\eps i i+1} =  - b^{i+1}_\eps$ for 
any $i = 1, \cdots, N-2$
and
\begin{equation}
A^h_{\eps i i } = - (A^h_{\eps i i -1} + A^h_{\eps i i+1}), \quad i = 1, 
\cdots, N-1.
\label{eq:Ansum}
\end{equation}
In other words, the $i$th diagonal entry of $A^h_\eps$ is the negative sum 
of its neighbors in each row. Here, $A^h_{\eps 01}$ and $A^h_{\eps N-1 N}$ 
are not matrix elements and are set to be $b^1_\eps$ and $b^N_\eps$, 
respectively.

\item[(N3)] On each interval $I_j$ for $j=1, \cdots, N$, the only
two basis functions that are nonzero are $\phi^j_\eps$ and 
$\phi^{j-1}_\eps$, and they sum to one, i.e., $\phi^j_\eps + 
\phi^{j-1}_\eps = 1$. Equivalently, we have
\begin{equation}
\phi^j_\eps(x) = {\bf 1}_{I_j} \tilde{\phi}^j_\eps(x) + {\bf 
1}_{I_{j+1}}(x)[1-\tilde{\phi}^{j+1}(x)],
\label{eq:msfembdec}
\end{equation}
for some functions $\{\tilde{\phi}^k_\eps(x)\}_{k=1}^N$, each of them 
defined only on $I_j$ with boundary value 0 at the left end point and 1 at 
the right.
\end{enumerate}

As we shall see for MsFEM, (N3) implies (N2) when the bilinear form is 
symmetric. The special tri-diagonal structure of $A^h_\eps$ in (N2) 
simplifies the calculation of its action on a vector. Let $U$ be any vector 
in $\R^{N-1}$, we have
\begin{equation}
(A^h_\eps U)_i = - D^+( b^i_\eps D^- U_i), \quad\quad i = 1, \cdots, N - 1.
\label{eq:conserve}
\end{equation}
Here, the symbol $D^-$ denotes the backward difference operator, which is defined, together with the forward difference operator $D^+$, as
\begin{equation}
(D^- U)_k = U_k - U_{k-1}, \quad (D^+ U)_k = U_{k+1} - U_k.
\label{eq:fbdo}
\end{equation}
The equality \eqref{eq:conserve} is easy to check, and to make sense of the case when
$i$ equals $1$ or $N$, we need to extend the definition of $U$ by setting $U_0$ and
$U_N$ to zero. This formula has been used, for example, in \cite{HWC-MC-99}. It
is a very useful tool in the forthcoming computations.

\subsection{The expression of the corrector}

Now we derive an expression of the corrector, i.e., the difference between 
$u^h_\eps$, the solution to \eqref{eq:rode} obtained from the above scheme, 
and $u^h_0$, the standard FEM solution to \eqref{eq:hode}. 

The function $u^h_0(x)$ is obtained from a weak formulation similar to \eqref{eq:wfmsfem} with $a_\eps$ replaced by $\astar$, and $H_0^1$ replaced by $V^h_0$, the space spanned by hat functions $\{\phi^j_0\}$. Clearly, these basis functions satisfy (N1) and (N3). Let $A^h_0$ denote the associated stiffness matrix; then one can verify that it satisfies (N2). In fact, the vector $b$ is given by $b^k_0 = \astar/h$. Now $u^h_0(x)$ is simply $U^0_j \phi^j_0$, where $U^0$ solves
\begin{equation} A^h_0 U^0 = F^0.
\label{eq:feml}
\end{equation}
Subtracting this equation from \eqref{eq:msfeml}, we obtain:
$$
A^h_0(U^\eps - U^0) = (F^\eps - F^0) - (A^h_\eps - A^h_0)U^\eps.
$$
Let $G^h_0$ denote the inverse of the matrix $A^h_0$. We have
$$
U^\eps - U^0 = G^h_0(F^\eps - F^0) - G^h_0(A^h_\eps - A^h_0)U^\eps.
$$
Since both $A^h_\eps$ and $A^h_0$ satisfy (N2), the difference $A^h_\eps - A^h_0$ acts on vectors in the same manner as in \eqref{eq:conserve}. Since both $\{\phi^j_\eps\}$ and $\{\phi^j_0\}$ satisfy (N3), we verify that
$$
(F^\eps - F^0)_j = - D^+ (\tilde{F}^\eps_j - \tilde{F}^0_j), \quad\quad
\tilde{F}^\eps_j := \int_{I_j} f(t) \tilde{\phi}^j_\eps(t) dt.
$$
Using these difference forms, we have
\begin{equation}
\begin{aligned}
\label{eq:dcrtex}
U^\eps_j - U^0_j & = - \sum_{m = 1}^{N-1} (G^h_0)_{jm} 
\left(D^+(\tilde{F}^\eps - \tilde{F}^0)_m  - D^+\left( (b^m_\eps - b^m_0) 
D^- U^\eps_m \right) \right) \\
& = \sum_{k  = 1}^{N} D^- G^h_{0 j k} \left( (\tilde{F}^\eps - 
\tilde{F}^0)_k - (b^k_\eps - b^k_0) D^- U^\eps_k \right).
\end{aligned}
\end{equation}
The second equality is obtained from summation by parts. Note that we have
extended the definitions of $U^\eps$ and $U^0$ so that they equal zero when
the index is $0$ or $N$. Similarly, $(G^h_0)_{j0}$ and $(G^h_0)_{jN}$ are
zero as well.

The vector $U^\eps - U^0$ is the corrector evaluated at the nodal points. We have the following control of its $\ell^2$ norm under some assumptions on the statistics of $\{b^k_\eps\}$ and $\{\phi^j_\eps\}$.
\begin{proposition}
\label{prop:msfemell2}
Let $U^\eps$ and $U^0$ be as above. Let the basis functions $\{\phi^j_\eps\}$
and the stiffness matrix $A^h_\eps$ satisfy {\upshape (N1)-(N3)}. Suppose also
that
\begin{equation}
\sup_{1 \le k \le N} |D^- U^\eps_k| \le C \|f\|_{2} h^{\frac 1 2},
\label{eq:holder}
\end{equation}
for some universal constant $C$.

{\upshape (i)} Suppose further
that for any $k = 1, \cdots, N$, and any $x \in I_k$, we have
\begin{equation}
\label{eq:phibdd}
\E \left(\tilde{\phi}^k_\eps(x) - \tilde{\phi}^k_0(x)\right)^2 \le C 
\frac{\eps}{h}\|R\|_{1,\R},
\end{equation}
and
\begin{equation}
\label{eq:bbdd}
\E \left(b^k_\eps - b^k_0 \right)^2 \le C \frac{\eps}{h^3}\|R\|_{1,\R},
\end{equation}
for some  universal constant $C$. Then we have
\begin{equation}
\label{eq:msfemell2}
\E \left\| U^\eps - U^0 \right\|_{\ell^2}^2 \le C 
\frac{\eps}{h^3}\|R\|_{1,\R}(1 + \|f\|_2),
\end{equation}
for some universal $C$.

{\upshape (ii)} Suppose instead that the right hand side of \eqref{eq:phibdd} is 
$C\left(\frac{\eps}{h}\right)^{\alpha}$, and the right hand side of \eqref{eq:bbdd}
is $C\frac{1}{h^2}\left(\frac{\eps}{h}\right)^{\alpha}$. Then the estimate 
in \eqref{eq:msfemell2} should be changed to
$C\frac{1}{h^2} \left(\frac{\eps}{h}\right)^{\alpha}$.
\end{proposition}

The assumption \eqref{eq:holder} 
essentially says that $u^h_\eps$ should
have a H\"{o}lder regularity. Suppose the weak formulation associated
to the multiscale scheme admits a unique solution $u^h_\eps$ such that 
$\|u^h_\eps\|_{H^1} \le C\|f\|_{2}$. Then by Morrey's inequality
\cite[p.266]{Evans}, $u^h_\eps \in C^{0,{\frac 1 2}}$ in one dimension.
Consequently, \eqref{eq:holder} holds.

For MsFEM, we have a better estimate: $|D^- U^\eps_k| \le C h$ due to a
super-convergence result; see \eqref{eq:holderplus}. Therefore, the 
estimate in \eqref{eq:msfemell2}
can be improved to be $C\frac{\eps}{h^2}$ in case (i) and
$C\frac{1}{h}\left(\frac{\eps}{h}\right)^{\alpha}$ in case (ii).

Item (i) of this proposition is useful when the random medium $a(x)$, or 
equivalently $q(x)$, has SRC, while item (ii) is useful in the case of LRC.  
The constant $C$ in the second item depends on $(\lambda, \Lambda)$, $f$ 
and $R_g$, but not on $h$.

\begin{proof}
To prove (i), we use a super-convergence result, which we prove in section \ref{sec:msfem1},  to get $|D^- G^h_{0jk}|\le Ch$.  Using this estimate 
together with \eqref{eq:holder} and \eqref{eq:dcrtex}, we have
$$
\E \left\lvert U^\eps_j - U^0_j\right\rvert^2 \le Ch \sum_{k = 1}^N 
\E\left\lvert \tilde{F}^\eps_k - \tilde{F}^0_k \right\rvert^2
+ Ch^2 \sum_{k = 1}^N \E\left\lvert b^k_\eps - b^k_0 \right\rvert^2.
$$
For the second term, we use \eqref{eq:bbdd} and obtain
\begin{equation}
\sum_{k = 1}^N \E\left\lvert b^k_\eps - b^k_0 \right\rvert^2 \le \sum_{k = 
1}^N  C \frac{\eps}{h^3}\|R\|_{1,\R} = C\frac{\eps}{h^4}\|R\|_{1,\R}.
\label{eq:bell2}
\end{equation}
For the other term, an application of Cauchy-Schwarz to the definition of 
$\tilde{F}^\eps$ yields
$$
\left\lvert \tilde{F}^\eps_k - \tilde{F}^0_k \right\rvert^2 \le 
\|f\|_{2, I_k}^2 \|\tilde{\phi}^k_\eps - \tilde{\phi}^k_0\|_{2, I_k}^2.
$$
Using \eqref{eq:phibdd}, we have
\begin{equation}
\label{eq:phiL2}
\E \|\tilde{\phi}^k_\eps - \tilde{\phi}^k_0\|_{2, I_k}^2 = \int_{I_k}
 \E \left(\tilde{\phi}^k_\eps - \tilde{\phi}^k_0\right)^2(x) dx \le 
C\frac{\eps}{h}\cdot h\|R\|_{1,\R} = C\eps\|R\|_{1,\R}.
\end{equation}
Therefore, we have
$$
\E \left\lvert U^\eps_j - U^0_j \right\rvert^2 \le \left(Ch \sum_{k = 1}^N 
\|f\|_{2, I_k} \eps + Ch^2 \frac{\eps}{h^4}\right)\|R\|_{1,\R} \le C 
\frac{\eps}{h^2} \|R\|_{1,\R}(1+\|f\|_2).
$$
Note that this estimate is uniform in $j$. Sum over $j$ to complete the 
proof of (i).

Proof of item (ii) follows in exactly the same way, using the corresponding 
estimates.
\end{proof}

Now, the corrector in this general multi-scale numerical scheme is:
\begin{equation}
\begin{aligned}
u^h_\eps(x) - u^h_0(x) & = U^\eps_j \phi^j_\eps(x) - U^0_j \phi^j_0(x) \\
& = (U^\eps - U^0)_j \phi^j_0(x) + U^0_j(\phi^j_\eps - \phi^j_0)(x) + 
(U^\eps - U^0)_j(\phi^j_\eps - \phi^j_0)(x).
\end{aligned}
\label{eq:crtex}
\end{equation}
We call the three terms on the right hand side $K_i(x)$, $i = 1, 2, 3$. Now
$K_1(x)$ is the piecewise interpolation of the corrector evaluated at the nodal
points; $K_2(x)$ is the corrector due to different choices of basis 
functions;
and $K_3(x)$ is much smaller due to the previous proposition and  
\eqref{eq:phibdd}. Our analysis shows that $K_1(x)$ and $K_2(x)$
contribute to the limit when $\eps\to0$ while $h$ is fixed, but only
a part of $K_1(x)$ contributes to the limit when $h\to0$.

Due to self-averaging effect which is made precise in Lemma \ref{lem:oint},
integrals of $q_\eps(x)$ are small. Therefore, our goal is to decompose the
above expression into two terms: a leading term which is an oscillatory integral
against $q_\eps$, and a remainder term which contains multiple oscillatory
integrals.
\begin{proposition}
\label{prop:msfemL}
Assume that $u^h_\eps$ is the solution to \eqref{eq:rode} obtained from a
multi-scale scheme, which satisfies {\upshape (N1)-(N3)} and has basis functions
$\{\phi^j_\eps\}$, and that $u^h_0$ is the solution of \eqref{eq:hode} obtained by
the standard FEM with hat basis functions $\{\phi^j_0\}$. Let $b_\eps$ and $b_0$
denote the vectors in {\upshape (N2)} of these methods. Suppose that 
\eqref{eq:holder} holds and that for any
$k = 1, \cdots, N$, we have
\begin{equation}
\begin{aligned}
\tilde{\phi}^k_\eps(t) - \tilde{\phi}^k_0(t) & = 
[1 + \tilde{r}_{1k} ] \frac{\astar}{h}\left(\int_{x_{k-1}}^t q_\eps(s) ds - 
\frac{t-x_{k-1}}{h}\int_{x_{k-1}}^{x_k} q_\eps(s) ds \right),\\
b^k_\eps - b^k_0  & = [1 + \tilde{r}_{2k}]\left( -\frac{\astar^2}{h^2} \int_{x_{k-1}}^{x_k}
q_\eps(t) dt\right),
\end{aligned}
\label{eq:diffphib}
\end{equation}
for some random variables $\tilde{r}_{1k}$ and $\tilde{r}_{2k}$. 

{\upshape (i)} Assume that $q(x)$ has SRC, i.e., satisfies
{\upshape (S2)}, and that
\begin{equation}
\sup_{1 \le k \le N} \max \{ \E |\tilde{r}_{1k}|^2 ,  \E |\tilde{r}_{2k}|^2 \}
\le C\frac{\eps}{h}\|R\|_{1,\R},
\label{eq:phibrbdd}
\end{equation}
for some universal constant $C$. Then, the corrector can be written as
\begin{equation}
u^h_\eps(x) - u^h_0(x) = \int_0^1 L^h(x,t) q_\eps(t) dt + r^h_\eps(x).
\label{eq:crtdec}
\end{equation}
Furthermore, the remainder $r^h_\eps(x)$ satisfies
\begin{equation}
\sup_{x \in [0,1]} \ \E|r^h_\eps(x)| \le C\frac{\eps}{h^{5/2}} 
\|R\|_{1,\R}(1 + \|f\|_2),
\label{eq:rbdd}
\end{equation}
for some universal constant $C$. The function $L^h(x,t)$ is the sum of 
$L^h_1$ and $L^h_2$ defined by:
\begin{equation}
\begin{aligned}
L^h_1(x, t) & = \sum_{k = 1}^N  {\bf 1}_{I_k}(t) \frac{\astar D^- G^h_0(x, 
x_k)}{h} \left(\frac{\astar D^- U^0_k}{h} + \int_t^{x_k} f(s) ds -  
\int_{x_{k-1}}^{x_k} f(s) \tilde{\phi}^k_0(s) ds\right),
\\
L^h_2(x, t) & =  \frac{\astar}{h} D^- U^h_{0 j(x)} \left({\bf 
1}_{[x_{j(x)-1}, x]}(t) - \frac{x-x_{j(x) - 1}}{h} {\bf 
1}_{[x_{j(x)-1},x_{j(x)}]}(t)\right).
\label{eq:msfemL}
\end{aligned}
\end{equation}
Given $x$, the index $j(x)$ is the unique one so that $x_{j(x)-1} < x \le x_{j(x)}$. The function $G^h_0(x, x_k)$ is defined as
\begin{equation}
G^h_0(x, x_k) = \sum_{j = 1}^{N-1} G^h_{0 j k} \phi^j_0(x).
\label{eq:Gdef}
\end{equation}
$G^h_0$ is the interpolation in $V^h_0$ using the discrete Green's function of standard FEM.

{\upshape (ii)} Assume that $q(x)$ has long range correlation, i.e., 
{\upshape (L1)-(L3)} are satisfied, and that the estimate in 
\eqref{eq:phibrbdd} is $C\left(\frac{\eps}{h}\right)^{\alpha}$. Then the 
same decomposition holds, the expression of $L^h(x,t)$ remains the same, 
but the estimate in \eqref{eq:rbdd} should be replaced by $C 
\frac{1}{h^{3/2}}\left(\frac{\eps}{h}\right)^\alpha$.
\end{proposition}

Due to the super-convergent result in section \ref{sec:msfem1}, the function 
$G^h_0(x,x_k)$ above is exactly the Green's function evaluated at 
$(x,x_k)$. This can be seen from the facts that they agree at nodal points 
and both are piece-wise linear and continuous.
\begin{proof}
We only present the proof of item (i). Item (ii) follows in exactly the 
same way.  We point out that the assumption \eqref{eq:diffphib} and the
estimates \eqref{eq:phibrbdd} imply \eqref{eq:phibdd} and
\eqref{eq:bbdd} thanks to Lemma \ref{lem:oint}.

The idea is to extract the terms in the expression \eqref{eq:crtex} that 
are linear in $q_\eps$. For $K_1(x)$, we use \eqref{eq:dcrtex} and write
$$
\begin{aligned}
K_1(x) & \approx \sum_{j=1}^{N-1}\sum_{k=1}^{N} D^- G^h_{0 j k} 
(\tilde{F}^\eps_k - \tilde{F}^0_k - (b^k_\eps - b^k_0) D^- U^0_k) 
\phi^j_0(x)\\
& = \sum_{k=1}^N D^- G^h_0(x, x_k) (\tilde{F}^\eps_k - \tilde{F}^0_k - 
(b^k_\eps - b^k_0) D^- U^0_k).
\end{aligned}
$$
Note that the expression above is an approximation because we have changed 
$D^- U^\eps$ on the right hand side of \eqref{eq:dcrtex} to $D^- U^0$. The 
error is
\begin{equation}
r^h_{11}(x) = - \sum_{k=1}^N D^- G^h_0(x,x_k)(b^k_\eps - b^k_0)D^-(U^\eps - 
U^0)_k.
\label{eq:rh11}
\end{equation}
Estimating $|D^- G^h_0|$ by $Ch$ and using Cauchy-Schwarz on the sum over 
$k$ and \eqref{eq:bell2} and \eqref{eq:msfemell2}, we verify that 
$\E|r^h_{11}(x)| \le C\eps h^{-5/2}\|R\|_{1,\R}(1+\|f\|_2)$.

Using the expressions of $\tilde{\phi}_\eps$ and $b_\eps$, and the 
estimates of the higher order terms in them, \eqref{eq:diffphib}, we can further approximate 
$K_1(x)$ by
$$
\begin{aligned}
K_1(x) \approx \sum_{k=1}^N D^- G^h_0(x, x_k) & \left(\int_{x_{k-1}}^{x_k} 
f(t) \frac{\astar}{h}  \left[ \int_{x_{k-1}}^t q_\eps(s) ds -  
\frac{t-x_{k-1}}{h} \int_{x_{k-1}}^{x_k} q_\eps(s) ds\right]dt\right.
\\
+ & \left.\frac{\astar^2}{h^2}D^- U^0_k \int_{x_{k-1}}^{x_k} q_\eps(t) 
dt\right).
\end{aligned}
$$
The error in this approximation is:
\begin{equation}
\begin{aligned}
r^h_{12}(x) =  \sum_{k=1}^N D^- G^h_0(x, x_k) & \left(\tilde{r}_{1k}\int_{x_{k-1}}^{x_k} 
f(t) \frac{\astar}{h}  \left[ \int_{x_{k-1}}^t q_\eps(s) -  
\frac{t-x_{k-1}}{h} \int_{x_{k-1}}^{x_k} q_\eps(s)\right]dt\right.
\\
+ & \left. \tilde{r}_{2k} \frac{\astar^2}{h^2}D^- U^0_k \int_{x_{k-1}}^{x_k} q_\eps(t) 
dt\right).
\label{eq:rh12}
\end{aligned}
\end{equation}
Using Lemma \ref{lem:oint}, \eqref{eq:phibrbdd} and Cauchy-Schwarz, we have
$$
\E \Big|\tilde{r}_{1k} \int_{I_k} q_\eps(s) ds \Big| \le C \eps.
$$
Using this estimate, we verify that the mean of the absolute value
of the first term in $r^h_{12}$ is bounded by $C\eps \|f\|_2\|R\|_{1,\R}$. A similar
estimate with $|D^- U^0_k| \le Ch$ (due to super-convergence) shows that
the second term in $r^h_{12}$ has absolute mean bounded by $C\eps h^{-1}\|R\|_{1,\R}$.
Therefore, we have $\E|r^h_{12}(x)| \le C\eps 
h^{-1}(1+\|f\|_2)\|R\|_{1,\R}$.  We remark also that in the case of long 
range correlations, we should apply Lemma \ref{lem:ointl} instead.

Moving on to $K_2(x)$, we observe that for fixed $x$, $K_2(x)$ reduces to a 
sum over at most two terms, due to the fact that $\phi^j_\eps$ and 
$\phi^j_0$ have local support only.
Let $j(x)$ be the index so that $x \in (x_{j(x)-1},x_{j(x)}]$. We have
$$
\begin{aligned}
K_2(x) & = \sum_{j = 1}^N D^- U^0_j (\tilde{\phi}^j_\eps - 
\tilde{\phi}^j_0)(x) = D^- U^0_{j(x)}(\tilde{\phi}^{j(x)}_\eps - 
\tilde{\phi}^{j(x)}_0)(x)\\
& \approx D^- U^0_{j(x)} \frac{\astar}{h} \left(\int_{x_{j(x) - 1}}^{x} 
q_\eps(t) dt
- \frac{x - x_{j(x) - 1}}{h} \int_{x_{j(x) - 1}}^{x_{j(x)}} q_\eps(t) 
dt\right).
\end{aligned}
$$
In the second step above, we used the decomposition of 
$\tilde{\phi}_\eps$ again. The error we make in this step is
\begin{equation}
r^h_2(x) = \tilde{r}_{1j(x)} \frac{\astar D^- U^0_{j(x)}}{h} 
\left(\int_{x_{j(x) - 1}}^{x} q_\eps(t) dt
- \frac{x - x_{j(x) - 1}}{h} \int_{x_{j(x) - 1}}^{x_{j(x)}} q_\eps(t) 
dt\right).
\label{eq:rh2}
\end{equation}
We verify again that $\E|r^h_2(x)| \le C\eps\|R\|_{1,\R}$.

Now for $K_3(x)$, we use Cauchy-Schwarz and have
\begin{equation}
\E |K_3(x)| \le  \E \left(\sum_{j = 1}^{N-1} |U^h_{\eps j} - 
U^h_{0j}|^2\right)^{\frac 1 2} \left( \sum_{j = 1}^{N-1} (\phi^j_\eps - 
\phi^j_0)^2(x)\right)^{\frac 1 2}
\le C \frac{\eps}{h^2}\|R\|_{1,\R}(1+\|f\|_2).
\label{eq:K3bdd}
\end{equation}
The last inequality is due to \eqref{eq:msfemell2} and \eqref{eq:phibdd}.

In the approximations of $K_1(x)$ and $K_2(x)$, we change the order
of summation and integration. We find that $K_1(x)$ is then $\int_0^1 
L^h_1(x,t)q_\eps(t)dt$ plus the error term $r^h_{11}+r^h_{12}$, and 
$K_2(x)$ is $\int_0^1 L^h_2(x,t) q_\eps(t) dt$ plus the error term $r^h_2$.  
Therefore we proved \eqref{eq:crtdec} with $r^h_\eps(x) = r^h_{11}+r^h_{12} 
+ r^h_2 + K_3(x)$. The estimates above for these
error terms are uniform in $x$, verifying \eqref{eq:rbdd}.
\end{proof}

\subsection{Weak convergence of the corrector of a multiscale scheme}
\label{sec:wcc}
In this section, we characterize the limit of the corrector $u^h_\eps - 
u^h_0$, with proper scaling, in the multiscale scheme when $\eps$ is sent 
to zero. As we have seen, the scaling depends on the correlation range of 
the random media.

\begin{proposition}
\label{prop:wcone}
Let $u^h_\eps$ be the solution to \eqref{eq:rode} given by a multi-scale 
scheme that satisfies {\upshape (N1)-(N3)}. Suppose \eqref{eq:holder} 
holds.  Let $u^h_0$ be the standard FEM solution to \eqref{eq:hode}.

{\upshape (i)} Suppose that $q(x)$ satisfies {\upshape (S1)-(S3)}
and that the conditions of item one in Proposition \ref{prop:msfemL}
hold. Then,
\begin{equation}
\frac{u^h_\eps - u^h_0}{\sqrt{\eps}} \xrightarrow[\eps \to 0]
{\mathrm{distribution}}  \sigma \int_0^1 L^h(x, t) dW_t.
\label{eq:wcone}
\end{equation}

{\upshape (ii)} Suppose that $q(x)$ satisfies {\upshape (L1)-(L3)}
and that the conditions of item two in Proposition \ref{prop:msfemL}
hold. Then,
\begin{equation}
\frac{u^h_\eps - u^h_0}{\eps^{\frac \alpha 2}} \xrightarrow[\eps \to 0]
{\mathrm{distribution}} \sigma_H \int_0^1 L^h(x, t) dW^H_t.
\label{eq:wconel}
\end{equation}
The real number $\sigma$ is defined in \eqref{eq:sigmadef} and
$\sigma_H$ is defined in Theorem \ref{thm:msfeml}.
\end{proposition}

These results allow us to prove the weak convergence in step (ii) of the diagram in \eqref{eq:cd} for fairly general schemes. A standard method to attain such
weak convergence results is to use the following Proposition; see 
\cite[p.64]{KS-BMSC}.

\begin{proposition}
Suppose $\{M_\eps\}$ with ${\eps \in (0,1)}$ is a family of random processes with values in the space of continuous functions $\Ct$
and $M_\eps(0) = 0$. Then $M_\eps$ converges in distribution to $M_0$
as $\eps \to 0$ if the following holds:


{\upshape (i) (Finite-dimensional distributions)} for any $0 
\le x_1\le \cdots \le x_k \le 1$, the joint distribution of $(M_\eps(x_1), 
\cdots, M_\eps(x_k))$ converges to that of $(M_0(x_1), \cdots, M_0(x_k))$ 
as $\eps \to 0$.


{\upshape (ii) (Tightness)} The family $\{M_\eps\}_{\eps \in (0,1)}$ is a 
tight sequence of random processes in $\Ct(I)$. A sufficient condition is 
the Kolmogorov criterion: $\exists \delta, \beta, C>0$ such that
\begin{equation}
\E\left\{ \left\lvert M_\eps(s) - M_\eps(t)\right\rvert^{\beta} \right\} 
\le C|t-s|^{1+\delta},
\label{eq:kolmogorov}
\end{equation} 
uniformly in $\eps$ and $t, s \in (0,1)$.
\label{prop:wc}
\end{proposition}

  The standard Kolmogorov criterion for tightness requires the existence of $t \in 
  [0,1]$ and some exponent $\nu$ so that $\sup_{\eps} \E|M_\eps(t)|^\nu \le 
  C$ for $C$ independent of $\eps$ and $\nu$. In our 
  cases, since $M_\eps(0) = 0$ for all $\eps$, this condition is always 
  satisfied.

We will prove item one of Proposition \ref{prop:wcone} in detail; proof of
item two follows in the same way, so we only point out the necessary
modifications. Recall the decomposition in \eqref{eq:crtdec}. Let $I_\eps$
denote the first member on the right hand side of this equation, i.e., the
oscillatory integral. Let $\I^h$ denote the right hand side of 
\eqref{eq:wcone}. The strategy in the case of SRC is to show that
$\{\eps^{-\frac 1 2} I_\eps \}$ converges in distribution in $\Ct$ to the
target process $\I^h$, while $\{\eps^{-\frac 1 2} r^h_\eps\}$ converges
in distribution in $\Ct$ to the zero function. Since the zero process is 
deterministic, the convergence in fact holds in probability; see 
\cite[p.27]{B-CPM}. Then \eqref{eq:wcone}
follows.

\begin{proof} {\em Convergence of $\{\eps^{-\frac 1 2} I_\eps\}$}. We first
check that the finite dimensional distributions of $I_\eps(x)$ converge
to those of $\I^h(x)$. Using characteristic functions, this amounts to showing
\begin{equation*}
\E \exp\left(\frac{i}{\sqrt{\eps}} \int_0^1 q_\eps(t) 
\textstyle\sum_{j=1}^n \xi_j L^h(x^j, t) dt\right) \xrightarrow{\eps \to 0} 
\E \exp \left( i \sigma \int_0^1 \textstyle \sum_{j=1}^n \xi_j L^h(x^j,t) 
dW_t\right),
\end{equation*}
for any positive integer $n$, and any $n$-tuples $(x^1, \cdots, x^n)$
and $(\xi_1, \cdots, \xi_n)$. We set $m(t) = \sum_{j = 1}^n \xi_j L^h(x^j, 
t)$. The convergence above is proved if we can show
\begin{equation}
    \frac{1}{\sqrt{\eps}} \int_0^1 q_\eps(t) m(t) dt \xrightarrow[\eps \to 
    0]{\mathrm{distribution}} \sigma\int_0^1 m(t) dW_t,
    \label{eq:cltbal}
\end{equation}
for any $m(t)$ that is square integrable on $[0,1]$.  Indeed, this
convergence holds as long as $q(x,\omega)$ is a stationary 
mean-zero process that admits an integrable $\rho$-mixing coefficient $\rho(r) 
\in L^1(\R)$. This result is more or less standard, and a proof can be found 
in \cite[Theorem 3.7 and its proof]{B-CLH-08}. Our assumptions (S1)-(S3)
guarantee the existence of such a $\rho(r)$. Therefore, we proved the
convergence of the finite distributions of $\{\eps^{-\frac 1 2} I_\eps\}$.

Next, we establish tightness of $\{\eps^{-\frac 1 2} I_\eps(x)\}$ by verifying
\eqref{eq:kolmogorov}. Consider the fourth moments and recall $L^h = L^h_1
+ L^h_2$ in \eqref{eq:msfemL}; we have
\begin{equation}
    \begin{aligned}
      \E (I_\eps(x) - I_\eps(y))^4 \le 8 & \Big\{\E 
      \Big(\frac{1}{\sqrt{\eps}} \int_0^1 q_\eps(t) (L^h_1(x,t) - 
      L^h_1(y,t)) dt \Big)^4\\
      & \quad + \E \Big(\frac{1}{\sqrt{\eps}} \int_0^1 q_\eps(t) 
      (L^h_2(x,t) - L^h_2(y,t)) dt \Big)^4\Big\}.
    \end{aligned}
    \label{eq:dLh12}
  \end{equation}
  We estimate the two terms on the right separately. For the first term we 
  observe that $L^h_1(x,t)$ is Lipschitz continuous in $x$. This is due to  
  the fact that $G^h_0(x,x_k)$ is Lipschitz in $x$ with a universal 
  Lipschitz coefficient. Since the other terms in the expression of 
  $L^h_1(x,t)$ in \eqref{eq:msfemL} are bounded by $C$, we 
  have
  \begin{equation*}
    |L^h_1(x,t) - L^h_2(y, t)| \le \frac{C}{h} |x - y|.
  \end{equation*}
  We use this fact and apply Lemma 
  \ref{lem:fmtint} to deduce
  \begin{equation}
    \E \Big(\frac{1}{\sqrt{\eps}}\int_0^1 q_\eps(t) (L^h_1(x,t) - 
    L^h_1(y,t)) dt \Big)^4 \le \frac{C}{h^4} |x-y|^4.
    \label{eq:fLh1}
  \end{equation}
  The constant $C$ above depends on $\lambda$, $\Lambda$, and 
  $\|\rho^{\frac 1 2}\|_{1, \R_+}$.
  
  To estimate the second term in \eqref{eq:dLh12}, consider two distinct 
  points $y<x$. Let $j$ and $k$ be the indices such 
  that $x \in (x_{j-1}, x_j]$ and $y \in (x_{k-1}, x_k]$. Then one of the 
  following holds: $j-k \ge 2$, $j-k = 0$ or $j-k = 1$. In the first case, 
since $|D^- U^0| \le Ch$ for some $C$ depending on $\lambda, \Lambda$ and 
$\|f\|_2$, we have the following crude bound.
  $$
    |L^h_2(x, t) - L^h_2(y, t)| \le C \le \frac{C}{h}|x - y|.
  $$
  The same analysis leading to \eqref{eq:fLh1} applies, and the second term 
  in \eqref{eq:dLh12} is bounded by $C|x-y|^4/h^4$ in this case.  
  
  When $|j-k| = 0$, $x$ and $y$ are in the same interval $(x_j, x_{j+1})$.  
  We can write
  \begin{equation}
    \int_0^1 q_\eps(t) (L^h_2(x,t) - L^h_2(y,t)) dt = \frac{\astar D^- 
    U^h_{0j}}{h} \left( \int_y^x q_\eps(t) dt - \frac{x-y}{h}\int_{I_j} 
    q_\eps(t) dt \right).
    \label{eq:jek}
  \end{equation}
  Since $x$ and $y$ are in the same interval, the function $(x-y)/h$
  is bounded by one. Now Lemma \ref{lem:fmtint} applies and we see
  that the fourth moments of the members in \eqref{eq:jek} are bounded by
  $$
  C\left[\E\left(\frac{1}{\sqrt{\eps}} \int_x^y q_\eps(t)dt\right)^4 +
  \left(\frac{x-y}{h}\right)^4 \E\left(\frac{1}{\sqrt{\eps}} \int_{I_k} q_\eps(t)dt\right)^4\right] \le C|x-y|^2.
  $$
  When $j-k=1$, we have
  \begin{equation*}
    \begin{aligned}
      \int_0^1 q_\eps(t) L^h_2(y,t) dt &= \frac{\astar D^- U^h_{0j-1}}{h} 
      \left(\int_{x_{j-2}}^y q_\eps(t) dt - 
      \frac{y-x_{j-2}}{h}\int_{x_{j-2}}^{x_{j-1}} q_\eps (t) dt \right)\\
      & = \frac{\astar D^- U^h_{0j-1}}{h} \left( - \int_y^{x_{j-1}} 
      q_\eps(t) dt - \frac{y-x_{j-1}}{h}\int_{x_{j-2}}^{x_{j-1}} q_\eps (t) 
      dt
      \right).
    \end{aligned}
  \end{equation*}
  Let $x_{j-1}$ play the role of $x$ in \eqref{eq:jek} and notice that
  $L^h_2(x_{j-1},t) = 0$.  We get
  \begin{equation*}
    \E \left( \frac{1}{\sqrt{\eps}} \int_0^1 q_\eps(t) L^h_2(y,t) dt \right)^4 
    \le C\frac{|y - x_{j-1}|^2}{h^2}.
  \end{equation*}
  Similarly, in the interval where $x$ lands, let $x_{j-1}$ play the role
  of $y$ in \eqref{eq:jek}. We have
  \begin{equation*}
    \E \Big( \frac{1}{\sqrt{\eps}} \int_0^1 q_\eps(t) L^h_2(x,t) dt \Big)^4 
    \le C\frac{|x - x_{j-1}|^2}{h^2}.
  \end{equation*}
  We combine these estimates and see that in this case, the second term in 
  \eqref{eq:dLh12} is bounded by
  \begin{equation*}
    \begin{aligned}
      & 8 \Big(\E \left( \frac{1}{\sqrt{\eps}} \int_0^1 q_\eps(t) 
      L^h_2(y,t) dt \Big)^4 + \E \Big( \frac{1}{\sqrt{\eps}} \int_0^1 
      q_\eps(t) L^h_2(x,t) dt \Big)^4 \right)\\
      \le & C\frac{|x_{j-1} - y|^2 + |x - x_{j-1}|^2}{h^2} \le C 
      \frac{|x-y|^2}{h^2}.
    \end{aligned}
  \end{equation*}
  In the last inequality, we used the fact that $a^2 + b^2 \le (a+b)^2$ for two non-negative numbers $a$ and $b$. 

  Combine these three cases to conclude that for any $x, y \in [0,1]$, the 
  second term in \eqref{eq:dLh12} is bounded by $C|x-y|^2/h^2$. This, together
  with \eqref{eq:fLh1}, shows
  \begin{equation}
    \E \Big( \frac{1}{\sqrt{\eps}} \int_0^1 q_\eps(t) (L^h(x,t) - L^h(y,t)) 
    dt \Big)^4 \le C \frac{|x-y|^2}{h^4}.
  \end{equation}
  In other words, $\{\eps^{-\frac 1 2}I_\eps(x)\}$ satisfies \eqref{eq:kolmogorov} with 
  $\beta = 4$ and $\delta = 1$, and is therefore a tight sequence. Consequently, it
  converges to $\I^h$ in distribution in $\Ct$.
    
{\em Convergence of $\{\eps^{-\frac 1 2} r^h_\eps\}$}.
For the convergence of finite dimensional distributions, we need to show
\begin{equation*}
\E \exp \left(i\cdot \frac{1}{\sqrt{\eps}} \textstyle \sum_{j = 1}^n \xi^j 
r^h_\eps(x^j) \right) \to 1,
\end{equation*}
for any fixed $n$, $\{x^j\}_{j=1}^n$ and $\{\xi_j\}_{j=1}^n$. Since 
$|e^{i\theta} - 1|\le |\theta|$ for any real number $\theta$, the left hand side of the equation above can be bounded by
\begin{equation*}
  \frac{1}{\sqrt{\eps}} \E\Big|\textstyle \sum_j \xi_j r^h_\eps(x_j)\Big|
\le  \textstyle\sum_j |\xi_j| \displaystyle \frac{1}{\sqrt{\eps}} 
\sup_{1\le j \le n} \E|r^h_\eps(x_j)|.
\end{equation*}
The last sum above converges to zero thanks to \eqref{eq:rbdd}, completing 
the proof of convergence of finite dimensional distributions.

For tightness, we recall that $r^h_\eps(x)$ consists of $r^h_{11}$ in 
\eqref{eq:rh11}, $r^h_{12}$ in \eqref{eq:rh12}, $K_3(x)$ in 
\eqref{eq:crtdec} and $r^h_2(x)$ in \eqref{eq:rh2}. In the first three 
functions, $x$ appears in Lipschitz continuous terms, e.g., in
$D^-G^h_0(x;x_k)$ or $\phi^j_\eps(x)-\phi^j_0(x)$. Meanwhile, the
terms that are $x$-independent have mean square of order $O(\eps)$
or less. Therefore, we can choose $\beta = 2$ and $\delta = 1$ in
\eqref{eq:kolmogorov}. For instance, we consider $r^h_{12}(x)$ in
\eqref{eq:rh12}. Since $q_\eps$ is uniformly bounded, the integrals of 
$q_\eps$ on the interval $I_k$ are bounded by $Ch$. Recall that $|D^- 
U^0_k| \le Ch$ also; we have
\begin{equation*}
\left\lvert r^h_{12}(x) - r^h_{12}(y) \right \rvert \le \sum_{k=1}^N |D^- 
(G^h_0(x,x_k) - G^h_0(y,x_k))| \left(|\tilde{r}_{1k}| \int_{I_k} C|f| dt + 
C |\tilde{r}_{2k}| \right).
\end{equation*}
By the Lipschitz continuity of $D^- G^h_0$ and the estimate 
\eqref{eq:phibrbdd}, we have
\begin{equation*}
\E\left(\frac{r^h_{12}(x)-r^h_{12}(y)}{\sqrt{\eps}}\right)^2 \le 
C \frac{1}{\eps} |x-y|^2 \sup_{k}\E\{|\tilde{r}_{1k}|^2 
+|\tilde{r}_{2k}|^2\} \le C \frac{|x-y|^2}{h}.
\end{equation*}

Similarly, we can control $r^h_{11}$ and $K_3$. For $r^h_{2}$ in 
\eqref{eq:rh2}, we observe that it has the form of the main part of 
$K_2(x)$, which corresponds to $L^h_2(x,t)$ and the second term
in \eqref{eq:dLh12}, except the extra integral of $q_\eps$.
Therefore, the tightness argument for the second term in \eqref{eq:dLh12}
can be repeated. The extra $q_\eps$ term is favorable: we can choose 
$\beta = 2$ and $\delta = 1$ in \eqref{eq:kolmogorov}.

To summarize, $\{ \eps^{-\frac 1 2} r^h_\eps/\sqrt{\eps}\}$
can be shown to be tight by choosing $\beta = 2$ and $\delta = 1$ in 
\eqref{eq:kolmogorov}. Therefore,
it converges to the zero function in distribution in 
$\Ct$. We have thus established the convergence in \eqref{eq:wcone}.

\bigskip

{\em The case of LRC}. In this case, the scaling is
$\eps^{-\frac \alpha 2}$. The proof is almost the same as above
and we only point out the key modifications.

Let us denote the right hand side of \eqref{eq:wconel} by $\I^h_H$.
To show the convergence of the finite dimensional distributions of 
$\{ \eps^{- \frac \alpha 2} I_\eps\}$, instead of using \eqref{eq:cltbal},
we need the following analogue for random media with LRC:
\begin{equation}
\frac{1}{\eps^{\frac \alpha 2}} \int_0^1 q_\eps(t) m(t) dt 
\xrightarrow[\eps \to 0]{\mathrm{distribution}} \sigma_H \int_0^1 m(t) 
dW^H_t,
\label{eq:cltbg}
\end{equation}
where $\sigma_H$ is defined below \eqref{eq:cltonel}. The above holds
only for $q(x,\omega)$ constructed as in (L1)-(L3), and
$m \in L^1(\R)\cap L^\infty(\R)$. This convergence result was established
in Theorem 3.1 of \cite{BGMP-AA-08}. Assumptions on $q(x)$ in item
(ii) allows us to use this result and conclude that the finite dimensional distributions of
$\{\eps^{-\frac \alpha 2} I_\eps(x)\}$ converge to those of $\I^h_H$. 
For the tightness of $\{ \eps^{-\frac \alpha 2} I_\eps(x)\}$, we can follow the 
same procedures that lead to \eqref{eq:fLh1} and \eqref{eq:jek}. We only need to consider second order moments when applying the Kolmogorov criterion thanks to Lemma \ref{eq:ointl}, which says
\begin{equation}
\E\Big(\frac{1}{\eps^{\frac \alpha 2}} \int_x^y q_\eps(t) dt \Big)^2 \le 
C|x-y|^{2-\alpha}.
\label{eq:koll}
\end{equation}
In the SRC case, since $\alpha$ equals one we only have $|x-y|$ on the 
right.  To get an extra exponent $\delta$, we had to consider
fourth moments. In the LRC case, $\alpha$ is less than one, so
we gain a $\delta = 1-\alpha$ from the above estimate.
With this in mind, we can simplify the proof we did for \eqref{eq:wcone} to
prove that $\{\eps^{-\frac \alpha 2} I_\eps\} $ converges to $\I^h_H$.
Similarly, $\{ \eps^{-\frac \alpha 2} r^h_\eps\}$ converges to the zero
function in distribution, and hence in probability, in the space $\Ct$.
The conclusion is that \eqref{eq:wconel} holds. This completes the proof
of Proposition \ref{prop:wcone}.
\end{proof}

From the proofs of the propositions in this section, the results often hold
if the conditions in item (i) or (ii) of Proposition \ref{prop:msfemL}
are violated in an $\eps$-independent manner. For instance, if the
second equation in \eqref{eq:diffphib} is modified to
\begin{equation}
b^k_\eps - b^k_0   = c(h)[1 + \tilde{r}_{2k}]
\left( -\frac{\astar^2}{h^2} \int_{D_k}
q_\eps(t) dt\right),
\label{eq:diffbmod}
\end{equation}
for some function $c(h)$ and for region $D_k \subset I_k$, then this 
modification will be carried to $L^h(x,t)$ and following estimates, 
but the weak convergences in
Proposition \ref{prop:wcone} still hold.

\section{Weak convergence as $h$ goes to $0$}
\label{sec:wctwo}

In the previous section, we established weak convergence of
the corrector $u^h_\eps - u^h_0$ of a general multi-scale scheme
when the correlation length $\eps$ of the random medium
goes to zero while the discretization $h$ is fixed. In this section, we
send $h$ to zero, and characterize the limiting process. We aim to
prove the following statement.

\begin{proposition}
\label{prop:wctwo}
Let $L^h(x,t)$ be defined as in \eqref{eq:msfemL}. As $h$ goes to
zero, the Gaussian processes on the right hand sides of
\eqref{eq:wcone} and \eqref{eq:wconel} have the following
limits in distribution in $\Ct$:
\begin{equation}
\sigma \int_0^1 L^h(x,t) dW_t \xrightarrow[h\to 0]{\mathrm
{distribution}} \U(x; W),
\label{eq:wctwo}
\end{equation}
where $\U$ is the Gaussian
process in \eqref{eq:clttwo}. Similarly,
\begin{equation}
\sigma_H \int_0^1 L^h(x,t) dW^H_t \xrightarrow[h\to 0]{\mathrm
{distribution}} \U_H(x; W^H),
\label{eq:wctwol}
\end{equation}
where $\U_H$ is the Gaussian process in \eqref{eq:clttwol}.
\end{proposition}

We consider the case of SRC first. Recall
that $\I^h(x)$ denotes the left hand side of \eqref{eq:wctwo}. It
can be split further into three terms as follows. Let us first split
$L^h_1(x,t)$ into two pieces:
\begin{equation}
\begin{aligned}
L^h_{11}(x, t) & = \sum_{k=1}^N {\bf 1}_{I_k}(t) \frac{\astar D^- G^h_0(x, 
x_k)}{h}\cdot\frac{\astar D^- U^0_k}{h},\\
L^h_{12}(x, t) & = \sum_{k=1}^N {\bf 1}_{I_k}(t) \astar D^- G^h_0(x, x_k) 
\left(\frac{1}{h} \int_t^{x_k} f(s) ds - \frac{1}{h} \int_{x_{k-1}}^{x_k} 
f(s)\tilde{\phi}^k_0(s) ds\right).
\end{aligned}
\label{eq:Lsplit}
\end{equation}
Then define $\I^h_i(x)$ by
\begin{equation}
\I^h_i(x; W) = \sigma \int_0^1 L^h_{1i}(x,t) dW_t, \ i = 1, 2. \quad 
\I^h_3(x; W) = \sigma \int_0^1 L^h_2(x,t) dW_t.
\label{eq:Uhsplit}
\end{equation}
As it turns out, $\I^h_1(x; W)$ converges to the desired limit, while 
$\I^h_2(x;W)$ and $\I^h_3(x;W)$ converge to zero in probability.

\begin{proof}[Proof of \eqref{eq:wctwo}] {\em Convergence of $\{\I^h_1(x)\}$}.
By Proposition \ref{prop:wc}, we show the convergence of finite distributions
of $\{\I^h_1(x)\}$ and tightness. Since all processes involved are Gaussian, for 
finite dimensional distribution it suffices to consider the covariance 
function $R_1(x,y): = \E\{\I^h_1(x)\I^h_1(y)\}$. By the It\^o isometry of 
Wiener integrals, we have
\begin{equation*}
R^h_1(x, y) = \sigma^2 \int_0^1 L^h_{11}(x,t) L^h_{11}(y,t) dt.
\end{equation*}
For any fixed $x$,  $L^h_{11}(x,t)$, as a function of $t$,  is a piecewise 
constant approximation of $L(x,t)$. This is obvious
from the expression of $L(x,t)$ in \eqref{eq:Ldef2}.
Therefore, $L^h_{11}(x,t)$ converges to $L(x,t)$ in \eqref{eq:Ldef} 
pointwise in $t$. Meanwhile, $L^h_{11}$ is uniformly bounded as well.  
The dominant convergence theorem yields that for any $x$ and $y$,
\begin{equation}
\lim_{h \to 0} R^h_1(x,y) = \sigma^2 \int_0^1 L(x,t)L(y,t) dt = \E(\U(x;W) 
\U(y;W)).
\label{eq:Rh1conv}
\end{equation}
This proves convergence of finite dimensional distributions.

The heart of the matter is to show that $\{\U^h_1(x; W)\}$ is a tight sequence.
To this end, we consider its fourth moment
\begin{equation}
\E\Big(\I^h_1(x) - \I^h_1(y)\Big)^4 = \int_{[0,1]^4} \prod_{i=1}^4 
(L^h_{11}(x,t_i) - L^h_{11}(y,t_i)) \E \prod_{i=1}^4 dW_{t_i}.
\label{eq:Uh1fmt}
\end{equation}
Since increments in a Brownian motion are independent Gaussian random variables, we have
\begin{equation}
\E \prod_{i=1}^4 dW_{t_i} = [\delta(t_1 - t_2)\delta(t_3 - t_4) + 
\delta(t_1 - t_3)\delta(t_2 - t_4) + \delta(t_1 - t_4)\delta(t_2 - t_3)] 
\prod_{i=1}^4 dt_i.
\label{eq:gaussprod}
\end{equation}
Using this decomposition, and the fact that the $L^h_{11}$ is piecewise constant,
we rewrite the fourth moment above as three times of
\begin{equation*}
\left(\int_0^1 (L^h_{11}(x,t) - L^h_{11}(y,t))^2 dt\right)^2  = 
\left[\sum_{k=1}^N \left(\frac{\astar D^- (G^h_0(x,x_k) - 
G^h_0(y,x_k))}{h} \frac{\astar D^- U^0_k}{h} \right)^2 h \right]^2.
\end{equation*}
Hence, we need to control $\|L^h_{11}(x,\cdot)-L^h_{11}(y,\cdot)\|_{2}$.
Since $G^h_0$ is the 
Green's function associated to \eqref{eq:hode}, it admits expression \eqref{eq:green} as remarked below Proposition \ref{prop:msfemL}. Fix $y < x$, and let 
$j_1$ and $j_2$ be the indices so that $y \in (x_{j_1 - 1},x_{j-1}]$ and $x 
\in (x_{j_2-1}, x_{j_2}]$. Then we can split the above sum into three 
parts. In the first part, $k$ runs from one to $j_1 - 1$. In that case, 
both $x_k$ and $x_{k-1}$ are less than $y$. Formula \eqref{eq:green} says: 
$\astar (G^h_0(x,x_k)-G^h_0(y,x_k)) = x_k(y-x)$. Consequently,
\begin{equation}
\frac{\astar D^- (G^h_0(x,x_k) - G^h_0(y,x_k))}{h} = (y-x).
\label{eq:DGdh}
\end{equation}
Since $|D^- U^0_k/h|$ is bounded, we have
\begin{equation}
\sum_{k=1}^{j_1 - 1} \left(\frac{\astar D^- (G^h_0(x,x_k) - 
G^h_0(y,x_k))}{h}\right)^2 \left(\frac{\astar D^- U^0_k}{h}\right)^2 h \le 
C|x-y|^2 \sum_{k=1}^{j_1 - 1} h \le C|x-y|^2.
\label{eq:Uh1fmt2}
\end{equation}
Another part is $k$ running from $j_2+1$ to $N$. In that case, both $x_k$ 
and $x_{k-1}$ are larger than $x$. The above analysis yields the same bound  
for this partial sum.

The remaining part is when $k$ runs from $j_1$ to $j_2$. In this case, for 
some $k$, $x_k$ may end up in $(y,x)$, and we have to use different 
branches of \eqref{eq:green} when evaluating $G^h_0(x, x_k)$ and $G^h_0(y, 
x_k)$. Consequently, the cancellation of $h$ in \eqref{eq:DGdh} will not 
happen, and we need to modify our analysis. We observe that, due to the 
Lipschitz continuity of $G^h_0$ and boundedness of $|D^-U^0/h|$, we always 
have
\begin{equation}
\sum_{k=j_1}^{j_2} \left(\frac{\astar D^- (G^h_0(x,x_k) - 
G^h_0(y,x_k))}{h}\right)^2 \left(\frac{\astar D^- U^0_k}{h}\right)^2 \cdot 
h \le C\frac{|x-y|^2}{h^2}\sum_{k = j_1}^{j_2} h.
\label{eq:Uh1fmt3}
\end{equation}
If $j_2-j_1 \le 1$, the last sum above is then bounded by $2C|x-y|^2/h$. In 
this case, it is clear that $|x-y|\le 2h$; as a result, the sum above is 
bounded by $C|x-y|$.

If $j_2-j_1 \ge 2$, the above estimate will not help much if $j_2-j_1$ is 
very large. Nevertheless, since $|D^- G^h_0/h|$ is bounded by some 
universal constant $C$. We have
\begin{equation}
\sum_{k=j_1}^{j_2} \left(\frac{\astar D^- (G^h_0(x,x_k) - 
G^h_0(y,x_k))}{h}\right)^2 \left(\frac{\astar D^- U^0_k}{h}\right)^2 \cdot 
h \le C\sum_{k = j_1}^{j_2} h.
\label{eq:Uh1fmt32}
\end{equation}
Meanwhile, we observe that in this case
\begin{equation*}
\begin{aligned}
3|x-y| & \ge 3(x_{j_2 - 1} - x_{j_1}) = 3(j_2 - j_1 -1)h = (j_2 - j_1 + 1)h 
+ 2(j_2 - j_1 - 2)h\\
& \ge (j_2 - j_1 + 1)h.
\end{aligned}
\end{equation*}
Consequently, the sum in \eqref{eq:Uh1fmt32} is again bounded by $C|x-y|$.
Combining these estimates, we have
\begin{equation}
\|L^h_{11}(x,\cdot) - L^h_{11}(y, \cdot)\|_{2}^2 \le C|x-y|.
\label{eq:Uh1fmt5}
\end{equation}
It follows from the equation below \eqref{eq:gaussprod} that $\{\I^h_1(x)\}$ is a tight sequence and hence converges to $\U(x,W)$.

{\em Convergence of $\I^h_{12}$ to zero function}. 
For the finite dimensional distributions, we consider the covariance function 
$R^h_2(x,y) = \E\{\I^h_2(x)\I^h_2(y)\}$. By It\^o isometry,
\begin{equation}
  \sigma^2 \int_0^1 L^h_{12}(x,t)L^h_{12}(y,t) dt.
  \label{eq:tL2fd}
\end{equation}
Now from the expression of $L^h_{12}(x,t)$, \eqref{eq:Lsplit}, we see 
that $L^h_{12}(x,t)$ converges to zero point-wise in $t$ for any fixed $x$. Indeed, 
in the above expression, $|D^- G^h_0/h|$ is uniformly bounded while the integrals
of $f(s)$ and of $f(s)\tilde{\phi}^k_0(s)$ go to zero due to shrinking 
integration regions. Meanwhile, $L^h_{12}$ is also uniformly bounded.  
The dominated convergence theorem shows $R_2(x,y) \to 0$ for any $x$ and 
$y$, proving the convergence of finite dimensional distributions. The 
tightness of $\{\I^h_2(x)\}$ is exactly the same as $\{\I^h_1(x)\}$; that is to
say, the properties of $D^- G^h_0$ can still be applied. We conclude that
$\{\I^h_2(x)\}$ converges to zero.

{\em Convergence of $\I^h_3(x)$ to zero}. For the finite dimensional 
distributions, we observe that $L^h_2(x,t)$ is uniformly
bounded and for any fixed $x$,  it converges to zero point-wise in $t$, due
to shrinking of the non-zero interval $I_{j(x)}$. The covariance function of 
$\I^h_3(x)$, therefore, converges to zero, proving convergence of finite
dimensional distributions.

For tightness, we consider the fourth moment of $\I^h_3(x)-\I^h_3(y)$.
By \eqref{eq:gaussprod}, it equals
\begin{equation}
  \E(\I^h_3(x;W) - \I^h_3(y; W))^4 =  
  3\left(\int_0^1 (L^h_2(x,t) - L^h_2(y,t))^2 dt\right)^2.
\label{eq:Uh3fmt}
\end{equation}
Recalling the expression of $L^h_2(x,t)$ in \eqref{eq:msfemL},
it is non-zero only on an interval of size $h$ and is uniformly bounded.
Let $j(x)$ be the interval where $L^h_2(x)$ is non-zero, and similarly
define $j(y)$. Assume $y < x$ without loss of generality. 
Consider three cases: $j(x) = j(y)$, $j(y) = j(x)-1$, and $j(x) - j(y) \ge 2$.
In the first case, $x$ and $y$ 
fall in the same interval $[x_{j-1},x_j]$ for some index $j$. Then we 
have
\begin{equation*}
\int_0^1 (L^h_2(x,t) - L^h_2(y,t))^2 dt \le
C\int_0^1 \Big({\bf 1}_{[x,y]}(t) - \frac{x-y}{h}{\bf 1}_{I_j}(t) 
\Big)^2 dt.
\end{equation*}
This integral can be calculated explicitly; it equals:
\begin{equation*}
\begin{aligned}
& \int_0^1 {\bf 1}_{[x,y]}(t) - 2\frac{x-y}{h}{\bf 1}_{[x,y]} + 
\frac{(x-y)^2}{h^2} {\bf 1}_{I_j}(t) dt\\
= & (x-y) - 2\frac{x-y}{h}(x-y) + \frac{(x-y)^2}{h^2} h = (x-y)[1- 
\frac{x-y}{h}].
\end{aligned}
\end{equation*}
Since $|1-(x-y)/h| \le 1$ and $|D^- U^0_k/h| \le C$, the above quantity is bounded by $C|x-y|$.

In the second case, with $j$ the unique index so that $y\le x_j < x$ and 
using the triangle inequality, we have
\begin{equation*}
\|L^h_2(x,t) - L^h_2(y,t)\|^2_2 \le 2\left(\|L^h_2(x,t) - L^h_2(x_j,t)\|^2_2
+ \|L^h_2(x_j,t) - L^h_2(y,t)\|^2_2\right).
\end{equation*}
For the first term of the right hand side above, let $x_j$ play
the role of $y$ in the previous calculation. This term is bounded by
$C(x-x_j)$. Similarly, for the second term, let $x_j$ play the role of
$x$, and we bound this term by $C(x_j - y)$. Consequently, we can
still bound $\|L^h_2(x,\cdot) - L^h_2(y,t)\|_{2}^2$ by $C|x-y|$.

In the third case, we have $h \le |x-y|$. Meanwhile, since $L^h_2$
is uniformly bounded and is nonzero only on intervals of size $h$.
We have
\begin{equation*}
\|L^h_2(x,t) - L^h_2(y,t)\|^2_2 \le Ch \le C|x-y|.
\end{equation*}

Combining these three cases, the conclusion is:
\begin{equation}
\E(\I^h_3(x;W) - \I^h_3(y; W))^4 \le C |x-y|^2.
\label{eq:Uh3fmt2}
\end{equation}
This proves tightness and completes proof of the first
item of Proposition \ref{prop:wctwo}.
\end{proof}

  In the proof above, we used the fact that $G^h_0(x)$ defined in
  \eqref{eq:Gdef} is in fact the real Green's function defined in \eqref{eq:green}. 
  However, the analysis follows as long as $|D^-_k G^h_0(x,x_k)/h|$ is piecewise
  Lipschitz in $x$ with constant independent of $h$, and the total number of 
  pieces does not depend on $h$.

  The fact that $\I^h_2(x)$ and $\I^h_3(x)$ do not contribute to the limit is quite
  remarkable. It says the following. As long as the limiting distribution of the
  corrector $u^h_\eps - u^h_0$ is considered, the role of the multi-scale basis functions
  is mainly to construct the stiffness matrix, which is reflected by $\I^h_1(x)$; its roles
  in constructing the load vector $F^\eps$ and in assembling the global 
function, which are reflected in $\I^h_2(x)$ and $\I^h_3(x)$ respectively, are asymptotically not important.

Now, we prove the second part of Proposition \ref{prop:wctwo}.  
\begin{proof}[Proof of \eqref{eq:wctwol}] Recall that $\I^h_H(x)$ denotes
the left hand side of \eqref{eq:wctwol}. Using the same splitting
of $L^h_1$ in \eqref{eq:Lsplit}, we can split $\I^h_H$ into three pieces
$\I^h_{Hi}(x)$, $i = 1, 2, 3$, as in \eqref{eq:Uhsplit}. The only necessary modification
is to replace $\sigma$ by $\sigma_H$ and to replace the Brownian
motion $W_t$ by the fractional Brownian motion $W^H_t$. We show that 
$\I^h_{H1}(x)$ converges
to $\U_H$ while $\I^h_{H2}$ and $\I^h_{H3}$ converge to the zero function.

{\em Convergence of finite dimensional distributions}. 
For $\I^h_{H1}$, we consider the covariance matrix 
$R^h_{H1}(x,y)$ defined by $\E\{\I^h_{H1}(x)\I^h_{H1}(y)\}$. Using
the isometry \eqref{eq:lito}, we have
\begin{equation}
R^h_{H1}(x,y) = \kappa \int_0^1 \int_0^1 
\frac{L^h_{11}(x,t)L^h_{11}(y,s)}{|t-s|^\alpha}\  dt ds.
\label{eq:RH1}
\end{equation}
As before, the integrand in the above integral converges to 
$L(x,t)L(y,s)/|t-s|^{\alpha}$ for almost every $(t,s)$. Meanwhile, since 
$L^h_{11}$ is uniformly bounded, the integrand above is bounded by 
$C|t-s|^{-\alpha}$ which is integrable. The dominated convergence theorem 
then implies that $R^h_{H1}$ converges to the covariance function of 
$\U_H(x;W^H)$. The convergence of finite distributions of $\I^h_{H2}$
and $\I^h_{H3}$ are similarly proved.

{\em Tightness}. Due to LRC, we only need to consider
the second moments in \ref{eq:kolmogorov}. For $\{\I^h_{H1}\}$, we consider
\begin{equation*}
\E(\I^h_{H1}(x) - \I^h_{H1}(y))^2 = \kappa \int_{\R^2} \frac{(L^h_{11}(	
x,t)-L^h_{11}(y,t))(L^h_{11}(x,s) - L^h_{11}(y,s))}{|t-s|^{\alpha}} dt ds,
\end{equation*}
using again the isometry \eqref{eq:lito}. Now we claim that
\begin{equation}
\|L^h_{11}(x,t) - L^h_{11}(y,t)\|_{L^p_t} \le C|x-y|^{\frac 1 p},
\label{eq:L1hLp}
\end{equation}
for any $p\ge 1$. Indeed, for $p = 2$, this is shown in \eqref{eq:Uh1fmt5}; the analysis there actually shows also that the above holds for $p=1$. For $p = \infty$, this follows from the uniform bound on $L^h_{11}$. For other $p$, this follows from interpolation; see \cite[p.75]{LL}.

Now, we apply the Hardy-Littlewood-Sobolev lemma \cite[\S 4.3]{LL} to the
expression of the second moment above. We obtain the bound
\begin{equation*}
C(\alpha) \kappa \|L^h_{11}(x,\cdot) - L^h_{11}(y, \cdot)\|_{L^1} 
\|L^h_{11}(x,\cdot) - L^h_{11}(y, \cdot)\|_{L^{\frac{1}{1-\alpha}}}
\le C |x-y|^{2-\alpha}.
\end{equation*}
Therefore, the Kolmogorov criterion \eqref{eq:kolmogorov} holds with $\beta = 
2$ and $\delta = 1-\alpha$, proving tightness of $\{\I^h_{H1}\}$. 
Tightness of $\{\I^h_{H2}\}$ follows in the same way because 
$L^h_{12}$ has the same structure as $L^h_{11}$ as remarked before.
Tightness of $\{\I^h_{H3}\}$ follows from the same argument above and
the control on $\|L^h_2(x,\cdot)-L^h_2(y,\cdot)\|_2^2$ in the equation
above \eqref{eq:Uh3fmt2}.
This complete the proof of \eqref{eq:wctwol}.
\end{proof}

\section{Applications to MsFEM in random media}
\label{sec:msfem}

In this section, we prove Theorems \ref{thm:msfem} and \ref{thm:msfeml}
as an application of the general results obtained in the preceding two sections by verifying that the multiscale finite element method (MsFEM) satisfies the conditions of Proposition \ref{prop:wcone}.

\subsection{The multiscale basis functions}
\label{sec:msfem1}
We describe the MsFEM for \eqref{eq:rode} as a special case of the methods 
developed in \cite{HW-JCP-97, HWC-MC-99}. We verify that this scheme satisfies assumptions (N1)-(N3).

Recall that we have a uniform partition of the interval $[0,1]$ with nodal
points $\{x_k\}_{k=1}^N$, where $x_0 = 0$, and $x_N = 1$, and $I_k$
the $k$th interval $(x_{k-1},x_k)$. The standard hat basis functions
are denoted by $\{\phi^j_0(x)\}_{j=1}^{N-1}$. They span the space $V^h_0$. The
idea of MsFEM is to replace the hat basis functions by $\{\phi^j_\eps\}$
which are constructed as
\begin{equation}
\left\{
\begin{aligned}
&\L_\eps \phi^j_\eps(x) = 0,& & \quad x \in I_1\cup I_2\cup \cdots \cup 
I_{N-1},\\
&\phi^j_\eps = \phi^j_0,& & \quad  x \in \{x_k\}_{k=0}^N.
\end{aligned}
\right.
\label{eq:msfemb}
\end{equation}
Here $\L_\eps$ is the differential operator in \eqref{eq:rode}. Clearly,
$\phi^j_\eps$ has the same support as $\phi^j_0$ and thus satisfies (N1).
Note that the  $\{\phi^j_\eps\}$ are constructed locally
on independent intervals, and are suitable for parallel computing.

For any $k = 1, \cdots, N$, we observe that the only non-zero basis functions
are $\phi^k_\eps$ and $\phi^{k-1}_\eps$. Further, they sum up to one at the boundary
points $x_{k-1}$ and $x_k$. Since equation \eqref{eq:msfemb} is of linear
divergence form, we conclude that $\phi^k_\eps(x) + \phi^{k-1}_\eps(x) \equiv 1$
on the interval. This shows that MsFEM satisfies (N3). In fact, the functions
$\{\tilde{\phi}^k_\eps\}_{k=1}^N $ for MsFEM are constructed by solving
\eqref{eq:msfemb} on $I_k$ with boundary values zero at $x_{k-1}$ and one
at $x_k$. Once they are constructed, $\{\phi^j_\eps\}$ is given by
\eqref{eq:msfembdec}. We can solve $\phi^k_\eps$ analytically  and obtain that
\begin{equation}
\tilde{\phi}^j_\eps = b^j_\eps \int_{x_{j-1}}^x a_\eps^{-1}(t) dt, 
\quad\quad b^j_\eps = \left( \int_{I_j} a_\eps^{-1} (t) dt \right)^{-1}.
\label{eq:msfemphi}
\end{equation}
Consequently, (N1) and (N3) indicate that MsFEM also satisfies (N2).
To calculate the entries of the stiffness matrix $A^h_\eps$, we fix any
$i = 2, \cdots, N-2$, and compute
\begin{equation*}
 (A^h_\eps)_{i-1 i} = - \int_{I_i} a_\eps 
\left(\frac{d\tilde{\phi}^i_\eps}{dx}\right)^2 dx =  -\left(a_\eps 
\frac{d\tilde{\phi}^i_\eps}{dx}\right)^2 \int_{I_i} a_\eps^{-1}(s) ds = -b^i_\eps.
\end{equation*}
The last equality can be verified from the fact that
$\tilde{\phi}^i_\eps$ solves \eqref{eq:msfemb} and integration by
parts. For $i = 1$ and $N$, we verify that \eqref{eq:Ansum} holds
for $b^0_\eps$ and $b^N_\eps$ given by \eqref{eq:msfemphi}.

We record here a well-known super-convergence result: When dimension $d=1$,  the standard finite element method is 
super-convergent, in the sense that it yields exact values at nodal points. In our case, $u^h_0(x_k) = u_0(x_k)$, where $u_0$
solves \eqref{eq:hode} and $u^h_0$ is the FEM approximation. 
We observe that this property is preserved by MsFEM. Indeed, let $Pu_\eps$ 
be the projection of $u_\eps$ in $V^h_\eps$, i.e., $Pu_\eps = 
u_\eps(x_j)\phi^j_\eps(x)$. Then, using integrations by parts, \eqref{eq:msfemb}, and the fact that $Pu_\eps - u_\eps$ 
vanishes at nodal points, we have
$$
A_\eps(Pu_\eps, v) = A_\eps(u_\eps, v) = F(v), \quad \forall v \in 
V^h_\eps.
$$
Since the second equality is also satisfied by $u^h_\eps$, it follows that
$A_\eps(Pu_\eps - u^h_\eps,v) = 0$ for any $v$ in $V^h_\eps$. In 
particular, by choosing $v = Pu_\eps - u^h_\eps$ and by coersivity of 
$A_\eps(\cdot,\cdot)$, we conclude that $Pu_\eps = u^h_\eps$.  
The super-convergence result follows.

Several useful results follow from this super-convergent property.
First, $u^h_\eps(x)$ of MsFEM coincides with the true solution 
$u_\eps(x)$ at nodal points.
Note that $u_\eps$ can be explicitly solved analytically and that
$|u_\eps(x) - u_\eps(y)| \le C|x-y|$ for some universal $C$. We then have
\begin{equation}
|D^- U^\eps_k| = |u^h_\eps(x_k) - u^h_\eps(x_{k-1})| \le Ch.
\label{eq:holderplus}
\end{equation}
This improves the condition \eqref{eq:holder} in Proposition \ref{prop:msfemL}
and hence improves several subsequent estimates. Second, a fact which we have used extensively before, we have
$|D^- G^h_0| \le Ch$ and for any fixed $x_k$, $G^h_0(x; x_k)$
defined in \eqref{eq:Gdef} equals the continuous Green's function
$G_0(x,x_k)$ for \eqref{eq:hode}.
This is because the functions agree at the nodal points due to
super-convergence and they are both piece-wise
linear in $x$.

\subsection{Proof of Theorems \ref{thm:msfem} and \ref{thm:msfeml}}

Since MsFEM is a scheme that satisfies (N1)-(N3), in order to apply 
\eqref{eq:msfemell2} and \eqref{eq:msfemL} in previous propositions, we 
only need
to check that \eqref{eq:diffphib} and \eqref{eq:phibrbdd} hold.
\begin{lemma}
\label{lem:phibbdd}
Let $\tilde{\phi}^k_\eps$ and $b^k_\eps$ be the functions in {\upshape
(N1)-(N3)}
for MsFEM  defined in \eqref{eq:msfemphi}. Let 
$\tilde{\phi}^k_0$ and $b^k_0$ be the corresponding functions  for FEM.

{\upshape (i)} Suppose $a(x,\omega)$ and $q(x,\omega)$ satisfy
{\upshape (S1)-(S3)}. Then \eqref{eq:diffphib} and \eqref{eq:phibrbdd}
hold and the conclusion of item one in Proposition \ref{prop:msfemL} 
follows.

{\upshape (ii)} Suppose $a(x,\omega)$ and $q(x,\omega)$ satisfy
{\upshape (L1)-(L3)}. Then the conditions and hence the conclusions of the 
second item of Proposition \ref{prop:msfemL} hold.
\end{lemma}
\begin{proof} From the explicit formulas \eqref{eq:msfemphi}, we have
\begin{equation*}
b^k_\eps - b^k_0 = \left(\int_{I_k} \frac{1}{a_\eps} dt\right)^{-1} -
\left(\int_{I_k} \frac{1}{\astar} dt\right)^{-1} = - b^k_\eps \frac{\astar}{h}
\int_{I_k} q_\eps(t) dt.
\end{equation*}
Comparing with the second equation in \eqref{eq:diffphib}, we find that
it is satisfied with
\begin{equation}
\tilde{r}_{2k} : = b^k_\eps \int_{I_k} q_\eps(t) dt.
\label{eq:r2k}
\end{equation}
Similarly, we have
\begin{equation}
\tilde{\phi}^k_\eps(x) - \tilde{\phi}^k_0(x) =
b^k_\eps \left(\int_{x_{k-1}}^x q_\eps(s) ds
 - \dfrac{x-x_{k-1}}{h} \int_{x_{k-1}}^{x_k} q_\eps(s) ds\right).
\end{equation}
This shows again that \eqref{eq:diffphib} holds with $\tilde{r}_{1k}$
having the same expression as $\tilde{r}_{2k}$ defined above.
In \eqref{eq:r2k}, since $0\le b^k_\eps \le \Lambda h^{-1}$, we
can apply Lemma \ref{lem:oint} in the case of SRC or apply Lemma 
\ref{lem:ointl} in the case of LRC to conclude that $\E|\tilde{r}_{2k}|^2 
\le Ch^{-1}\eps$ in the first setting,
while $\E|\tilde{r}_{2k}|^2 \le  C(\eps h^{-1})^{\alpha}$ in the second 
setting.
This completes the proof.
\end{proof}

Note that the estimates \eqref{eq:phibdd} and \eqref{eq:bbdd}
follow directly from this lemma. Therefore, we can apply Proposition
\ref{prop:msfemell2} directly. Now we prove Theorem \ref{thm:msfem}.
Estimates \eqref{eq:fem} and \eqref{eq:supL2} do not follow from
Propositions \ref{prop:wcone} and \ref{prop:wctwo} directly and need
additional considerations.

\begin{proof}[Proof of Theorem \ref{thm:msfem}] {\em Finite element
analysis}. We have seen that $u^h_\eps$ super-converges to $u_\eps$. From \eqref{eq:rode} and \eqref{eq:msfemb},
we observe that
the following equation holds on $I_j$ for $j = 1, \cdots, N$:
\begin{equation}
\left\{
\begin{aligned}
& \L_\eps (u_\eps - u^h_\eps) = f, \quad & \text{in } I_j,\\
& u^h_\eps - u_\eps = 0, \quad & \text{on } \partial I_j.
\end{aligned}
\label{eq:lrode}
\right.
\end{equation}
Using the ellipticity of the coefficient and integrations by parts, we obtain
\begin{equation*}
\begin{aligned}
\lambda \vert u^h_\eps - u_\eps \vert_{H^1, I_j}^2 & \le \int_{I_j} a_\eps 
\frac{d}{dx}(u^h_\eps-u_\eps) \cdot \frac{d}{dx} (u^h_\eps - u_\eps) \ dx
= \int_{I_j} (u^h_\eps - u_\eps) \L_\eps(u^h_\eps - u_\eps) \ dx\\
& = \int_{I_j} f(x) (u^h_\eps - u_\eps)(x) \ dx
 \le \|f\|_{2, I_j} \|u^h_\eps - u_\eps\|_{2, I_j}.
\end{aligned}
\end{equation*}
Now recall that the Poincar\'{e}-Friedrichs inequality says that
\begin{equation}
\|u^h_\eps - u_\eps\|_{2, I_j} \le \frac{h}{\pi} \vert u^h_\eps - u_\eps 
\vert_{H^1, I_j}.
\label{eq:PF}
\end{equation}
Combining the inequalities above, we obtain
\begin{equation*}
\vert u^h_\eps - u_\eps \vert_{H^1, I_j} \le 
\frac{h}{\lambda \pi}\|f\|_{2, I_j}.
\end{equation*}
Taking the sum over $j$, we obtain the first inequality in \eqref{eq:fem}. To 
get the second inequality, we first apply the Poincare-Friedrichs inequality 
to the equation above to get
\begin{equation}
\|u^h_\eps - u_\eps\|_{2, I_j} \le \frac{h^2}{\lambda \pi^2} \|f\|_{2, 
I_j},
\label{eq:hlL2}
\end{equation}
and then sum over $j$. This completes the proof of \eqref{eq:fem}
in item one of the theorem.

{\em Energy norm of the corrector}. By energy norm, we mean the
$L^2(\Omega, L^2(D))$ norm. Recall the decomposition of the
corrector into $K_i(x)$ in \eqref{eq:crtex}. For $K_1(x)$,
we apply Cauchy-Schwarz to get the following bound for $|K_1|^2$
$$
 \textstyle \sum_i (U^\eps - U^0)^2_i \textstyle \sum_j (\phi^j_0(x))^2 \le  
\textstyle \sum_i (U^\eps_i - U^0_i)^2 \left(\textstyle \sum_j 
\phi^j_0(x)\right)^2
= \left\| U^\eps - U^0 \right\|_{\ell^2}^2.
$$
In the above derivation, we used the fact that $\phi^j_0(x)$ is 
non-negative, and $\sum_j \phi^j_0(x) \equiv 1$. Now we apply 
\eqref{eq:msfemell2} to control this term. The function $K_2(x)$,
as in the proof of Proposition \ref{prop:msfemL}, 
can be written as $D^- U^0_{j(x)}(\tilde{\phi}^{j(x)}_\eps - 
\tilde{\phi}^{j(x)}_0)$. Then from \eqref{eq:phibdd}, we have $\E|K_2(x)|^2 
\le C\eps\|R\|_{1,\R}$. 
For $K_3$, we have controlled $\E|K_3(x)|$ in \eqref{eq:K3bdd}. To control 
$\E|K_3(x)|^2$, we observe that $|K_3(x)| \le C\|f\|_2$. Note that all three
estimates concluded in the three steps are uniform in $x$.
Combining them, we complete the proof of \eqref{eq:supL2}.

{\em Convergence in distribution as $\eps$ to zero}. To prove item two of the
theorem, we apply \eqref{eq:wcone} of Proposition \ref{prop:wcone}.
We need to verify \eqref{eq:holder} in addition to \eqref{eq:diffphib}
and \eqref{eq:phibrbdd}, which we already verified in the previous lemma.
But this is implied by \eqref{eq:holderplus}, and hence we obtain 
\eqref{eq:cltone}.

{\em Convergence in distribution as $h$ to zero}. To prove \eqref{eq:clttwo},
we apply the first result in Proposition \ref{prop:wctwo}. This completes
the proof of the theorem.
\end{proof}

\begin{proof}[Proof of Theorem \ref{thm:msfeml}] In this case, the random
processes $q(x)$ and $a(x)$ are constructed by (L1)-(L3). To prove the estimate
in the energy norm, we follow the same steps as in the proof above, but use
item two of Proposition \ref{prop:msfemell2} to control the term $\|U^\eps 
- U^0\|_{\ell^2}^2$
in $K_1(x)$ and use Lemma \ref{lem:ointl} to control the terms in $K_2(x)$
and $K_3(x)$.

To obtain the results in \eqref{eq:cltonel} and \eqref{eq:clttwol}, we verify
the conditions in item two of Propositions \ref{prop:wcone} and \ref{prop:wctwo},
applying the second case in Lemma \ref{lem:phibbdd} and following
the steps in the previous proof. This completes the proof of the theorem.
\end{proof}

\section{Applications to HMM in random media}
\label{sec:hmm}

In this section, we adapt the general approach described in Sections
\ref{sec:stra} and \ref{sec:wctwo} to the case of heterogeneous multiscale
method (HMM).

\subsection{The heterogeneous multiscale method}

The goal of HMM is to approximate the large-scale properties of the 
solution to \eqref{eq:rode} without computing the homogenized coefficient 
first. Suppose we already know this effective coefficient, i.e., $\astar$ 
in our case. Then the large-scale solution $u_0$ can be solved by 
minimizing the functional
\begin{equation*}
I[u] : = \frac{1}{2} A_0(u,u) - F(u) = \frac{1}{2}\int_0^1  \astar 
\left(\frac{du}{dx}\right)^2\ dx - \int_0^1 f u \ dx.
\end{equation*}
In numerical methods, the first integral above can be computed by the 
following mid-point quadrature rule:
\begin{equation*}
A_0(u,u) \approx \sum_{j=1}^{N} \left(\astar(x^j) 
\frac{du}{dx}(x^j)\right)^2 h.
\end{equation*}
Here $x^j = (x_{j-1} + x_j)/2$ is the mid-point of $I_j$. In HMM,  $\astar$ 
is unknown, and the idea is to approximate $(u'\astar u')(x^j)$ by 
averaging in a local patch around the point $x^j$. For instance, we can 
take
\begin{equation*}
\left(\astar (x^j)\frac{du}{dx} (x^j)\right)^2  \approx \frac{1}{\delta} 
\int_{I^{\delta}_j} \left(a_\eps(s) \frac{d (\lft u)}{dx}(s)\right)^2 ds.
\end{equation*}
Here, $I^\delta_j$ denotes the interval $x^j+{\frac \delta 2}(-1,1)$, that 
is, the small interval centered in $I_j$ with length $\delta$.  The 
operator $\lft$ maps a function $w$ in $V^h_0$, i.e., the space spanned by 
hat functions, to the solution of the following equation:
\begin{equation}
\left\{ \begin{aligned}
&\L_\eps (\lft w) = 0, & & \quad x \in I^\delta_1\cup\cdots \cup 
I^\delta_{N-1},\\
&\lft w = w, & & \quad x \in \{\partial I^\delta_j\}_{j=1}^{N-1}.
\end{aligned}
\right.
\label{eq:lft}
\end{equation}
The idea here is to encode small-scale structures of the random media into 
the construction of the bilinear form. The key difference that 
distinguishes HMM and MsFEM is that the above equations are solved for HMM 
on patches $I^\delta_k$ that are smaller than $I_k$.  We check that $\lft$ 
is a linear operator; therefore, the following approximation of 
$A_0(\cdot,\cdot)$ is indeed bilinear:
\begin{equation}
A^\delta_\eps(w, v) : = \sum_{j=1}^{N} \frac{h}{\delta} \int_{I^\delta_j} 
a_\eps\ \frac{d(\lft w)}{dx}\frac{d(\lft v)}{dx}\ dx.
\label{eq:Adedef}
\end{equation}
With this approximation of the bilinear form, HMM consists finding
\begin{equation*}
u^{h,\delta}_\eps : = \underset{w \in V^h_0}{\mathrm{argmin}}\  \frac{1}{2} 
A_{\eps}^\delta (w, w) - F(w).
\end{equation*}
This variational problem is equivalent to solving $u^{h,\delta}_\eps = 
U^{\eps,\delta}_j\phi^j_0(x)$, where $U^{\eps,\delta}$ is determined by the 
linear system
\begin{equation}
A^{h,\delta}_\eps U^{\eps,\delta} = F^0.
\label{eq:hmml}
\end{equation}
Therefore, the above HMM can be viewed as a finite element method. The 
finite dimensional space here is $V^h_0$. Therefore HMM clearly satisfies 
(N1) and (N3). To check (N2), we calculate the associated stiffness matrix 
$A^{h,\delta}_\eps$. It has entries $A^\delta_\eps(\phi^i_0,\phi^j_0)$. From the defining equation \eqref{eq:lft}, we see that $\lft \phi^i_0$ is 
non-zero only on $I^\delta_i \cup I^\delta_{i+1}$, which implies that 
$A^{h,\delta}_\eps$ is again tri-diagonal.  Further, we verify that $\lft 
\phi^i_0 + \lft \phi^{i-1}_0 = 1$ on the interval $I^\delta_i$, which can 
be obtained from integrations by parts and which implies that 
$A^{h,\delta}_\eps$ satisfies \eqref{eq:Ansum}. Therefore, HMM satisfies 
(N2).

In fact, we can calculate the $b_\eps$ vectors. Let us consider the 
$(i-1,i)$th entry of $A^{h,\delta}_\eps$, where $i$ can be $2,\cdots, 
N-1$.  Since $(\lft \phi^{i-1}_0)' = -(\lft \phi^i_0)'$ on 
$I^\delta_i$, we have
\begin{equation*}
\left(A^{h,\delta}\right)_{\eps i-1 i} = - \frac{h}{\delta} 
\int_{I^\delta_i} a_\eps(s) \left(\frac{d(\lft \phi^i)}{dx}\right)^2 ds.
\end{equation*}
Now from \eqref{eq:lft}, we verify that $a_\eps (\lft \phi^i_0)'$ on 
$I^\delta_i$ is a constant given by
\begin{equation*}
c^\delta_i = \Big(\int_{I^\delta_i} a_\eps^{-1}(s) ds\Big)^{-1} \lft 
\phi^i_0\Big|_{x^i-{\frac \delta 2}}^{x^i + {\frac \delta 2}} = 
\Big(\int_{I^\delta_i} a_\eps^{-1}(s) ds\Big)^{-1} \frac{\delta}{h}.
\end{equation*}
Therefore, we have
\begin{equation}
\left(A^{h,\delta}\right)_{\eps i-1 i} = -(c^\delta_i)^2 \frac{h}{\delta} 
\int_{I^\delta_i}a_\eps^{-1} ds =  - 
\frac{\delta}{h}\left(\int_{I^\delta_i} a_\eps^{-1}(s) ds\right)^{-1} = :  
- b^i_{\eps,\delta}.
\label{eq:Aedex}
\end{equation}
We extend the 
definition of $b^i_{\eps, \delta}$ to the cases of $i=1$ and $i=N$, 
and check that the $(1,1)$th and $(N-1,N-1)$th entries of the stiffness 
matrix also satisfy \eqref{eq:Ansum}. In particular, the action of 
$A^{h,\delta}_\eps$ on a vector satisfies the conservative form as in 
\eqref{eq:conserve}.
In the sequel, to simplify notation, we 
drop the $\delta$ in the notations $A^{h,\delta}_\eps$, $U^{\eps,\delta}$ 
and $b^i_{\eps,\delta}$. 

The well-posedness of the optimization problem above, or equivalently of the 
linear system \eqref{eq:hmml} is obtained by Lax-Milgram. We show that 
the bilinear form $A^\delta_\eps(\cdot,\cdot)$ is continuous and coercive.
Consider two arbitrary functions $w = W_i\phi^i_0$ and $v  = 
V_j\phi^j_0$ in $V^h_0$. Then:
\begin{equation*}
A^h_\eps(w,v) =  W_i A^h_{\eps ij}V_j = -\textstyle \sum_i W_i D^+(b^i_\eps 
D^-V_i) = \textstyle \sum_i D^- W_i b^i_\eps D^- V_i.
\end{equation*}
Estimating the entries of vector $b_\eps$ by its infinity norm and use 
Cauchy-Schwarz, we obtain
\begin{equation*}
|A^h_\eps(w,v)| \le \left(\sup_{1\le i \le N} b^i_\eps \right)
\|D^-W\|_{\ell^2}\|D^-V\|_{\ell^2} \le \Lambda |w|_{H^1}|v|_{H^1}.
\end{equation*}
In the last inequality above, we used the fact that $\lambda h^{-1} \le 
b^i_\eps \le \Lambda h^{-1}$, which can be seen from its definition in 
\eqref{eq:Aedex} and the uniform ellipticity of $a_\eps$, and that 
$\|D^-W\|_{\ell^2} = |w|_{H^1}\sqrt{h}$ for $w\in V^h_0$.  This proves 
continuity.  Taking $w = v$, we have
\begin{equation*}
A^h_\eps(w,w) \ge \left(\inf_{1\le i \le N} b^i_\eps \right)
\|D^-W\|_{\ell^2}\|D^-W\|_{\ell^2} \ge \lambda |w|_{H^1}^2.
\end{equation*}
This proves coercivity. Therefore, by the Lax-Milgram theorem for the 
finite element space \cite[p.137]{QV}, there exists a unique 
$u^{h,\delta}_\eps \in V^h_0$ that solves the optimization problem.  Further, we 
have
\begin{equation}
|u^{h,\delta}_\eps|_{H^1} \le \frac{1}{\lambda} \sup_{w \in V^h_0} 
\frac{F(w)}{|w|_{H^1}} \le \frac{1}{\pi \lambda}\|f\|_2.
\label{eq:uhdH1}
\end{equation}
An immediate consequence  is that $|D^- U^\eps| \le C\sqrt{h}$ by the 
argument following Proposition \ref{prop:msfemell2}.

\subsection{Proof of Theorem \ref{thm:hmm}}

To prove Theorem \ref{thm:hmm}, we apply Proposition 
\ref{prop:msfemL} to write the corrector $u^{h,\delta}_\eps - u^h_0$ as an 
oscillatory integral plus a lower order term. To apply Propositions 
\ref{prop:wcone} and \ref{prop:wctwo} and obtain the weak convergences, we need to consider the difference
$b^k_\eps-b^k_0$ since $\tilde{\phi}^j_\eps=\tilde{\phi}^j_0$ in HMM.
From the expression of $b^k_\eps$ in \eqref{eq:Aedex}, we have
\begin{equation*}
b^k_\eps - b^k_0 = - b^k_\eps \frac{\astar}{\delta} \int_{I^\delta_k} 
q_\eps(t) dt = -(1+\tilde{r}_{2k}) \frac{h}{\delta} \frac{\astar^2}{h^2} 
\int_{I^\delta_k} q_\eps(t) dt,
\end{equation*}
where $\tilde{r}_{2k}$ is a random variable defined by
\begin{equation*}
\tilde{r}_{2k}  = -\frac{h}{\delta} b^k_\eps \int_{I^\delta_k} q_\eps(t) 
dt.
\end{equation*}
We verify that in the case of SRC, i.e., when $q(x)$ satisfies (S1)-(S3), 
we have
\begin{equation}
\E|\tilde{r}_{2k}|^2 \le C\frac{\eps}{\delta}\|R\|_{1,\R}, \quad 
\E\left(b^k_\eps - b^k_0\right)^2 \le C \frac{\eps}{h^2 \delta} 
\|R\|_{1,\R},
\label{eq:bbddh}
\end{equation}
for some universal constant $C$. Comparing this with \eqref{eq:bbdd} and 
\eqref{eq:phibrbdd}, we observe that the estimates have been 
multiplied by a factor $\frac{h}{\delta}$ in the HMM case. Similarly, it 
can be checked that in the case of LRC, i.e., when $q(x)$ satisfies 
(L1)-(L2), these estimates will be multiplied by a factor of 
$\left(\frac{\delta}{h}\right)^\alpha$. With these formulas at hand, we 
prove the third main theorem of the paper.

\begin{proof}[Proof of Theorem \ref{thm:hmm}]{\em Short range media and
amplification effect}.  In this case, the difference of $b^k_\eps - b^k_0$ 
and an estimate of it was captured in \eqref{eq:bbddh} and the equation 
above it.  We cannot apply Propositions 
\ref{prop:msfemL} and \ref{prop:wcone} directly.  However, due to the remark at the end of section \ref{sec:wcc}, similar conclusions still hold. The same procedure 
as in the proof of Proposition \ref{prop:msfemL} shows that the $L^h(x,t)$ 
function for HMM is:
\begin{equation}
L^{h,\delta}(x, t)  = \frac{h}{\delta}\sum_{k = 1}^N  {\bf 
1}_{I^\delta_k}(t) \frac{\astar D^- G^h_0(x, x_k)}{h} \frac{\astar D^- 
U^0_k}{h}.
\label{eq:hmmL}
\end{equation}
The first weak convergence in \eqref{eq:clth} holds with this definition of 
$L^{h,\delta}$ as an application of a modified version of Proposition 
\ref{prop:wcone}.  Indeed, the proof there works with $L^{h,\delta}$ 
playing the role of $L^h_{11}$. The tightness is still obtained from the function 
$D^- G^h_0$, and the factor $\frac{h}{\delta}$ does not play any role at 
this stage.

When $h$ goes to zero, we can follow the proof of Proposition 
\ref{prop:wctwo} to verify the second convergence in \eqref{eq:clth}.  
Indeed, tightness can be proved in exactly the same way. All that needs to be 
modified is the limit of the covariance function of $\U^{h,\delta}(x; W)$, 
which is defined to be $\sigma \int_0^1 L^{h,\delta}(x,t) dW_t$. This  
covariance function, by the It\^o isometry, is as follows:
\begin{equation}
\begin{aligned}
R^{h,\delta}(x,y) : & = \sigma^2 \int_0^1 L^{h,\delta}(x,t) 
L^{h,\delta}(y,t) dt\\
& = \sigma^2 \frac{h^2}{\delta^2} \sum_{k=1}^N \delta \frac{\astar D^- 
G^h_0(x;x_k)}{h} \frac{\astar D^- G^h_0(y;x_k)}{h}\left(\frac{\astar D^- 
U^0_k}{h}\right)^2.
\end{aligned}
\label{eq:Rhd}
\end{equation}
Recall the expression of $L^h_{11}$ in \eqref{eq:Lsplit}. We verify that 
the above quantity can be written as
\begin{equation*}
\sigma^2 \frac{h}{\delta} \int_0^1 L^h_{11}(x,t) L^h_{11}(y,t) dt.
\end{equation*}
Now the convergence in \eqref{eq:Rh1conv} implies that $R^{h,\delta}$ 
converges to the covariance function of $\sqrt{\frac{h}{\delta}}\U(x;W)$.  
This completes the proof of \eqref{eq:clth}.

{\em Long range media}. The expression for  $b^k_\eps - b^k_0$  are given above. Therefore, we can 
apply Propositions \ref{prop:msfemL} and \ref{prop:wcone} (with 
modifications), to show that as $\eps$ goes to zero while $h$ is fixed, the 
HMM corrector indeed converges to $\U^{h,\delta}_H(x;W^H)$ defined in 
\eqref{eq:clthl}.
When $h$ is sent to zero, we can follow the proof of Proposition 
\ref{prop:wctwo} and show that $\U^{h,\delta}_H$ converges in distribution 
to some Gaussian process. To find its expression, we calculate the covariance function of $\U^{h,\delta}_H$. Thanks to 
the isometry \eqref{eq:lito}, it is given by
\begin{equation}
R^{h,\delta}_H(x,y) : = \kappa \int_0^1\int_0^1 \frac{L^{h,\delta}(x,t) 
L^{h,\delta}(y,s)}{|t-s|^\alpha} dt ds.
\label{eq:RhdH}
\end{equation}
Using the expression of $L^{h,\delta}$, and the following short-hand 
notations:
\begin{equation*}
J_k(x) : = \frac{\astar D^- G^h_0(x;x_k)}{h}\frac{\astar D^- U^0_k}{h},
\end{equation*}
the covariance function can be written as
\begin{equation*}
\frac{\kappa h^2}{\delta^2}  \Big(\sum_{k=1}^N\sum_{m=1}^{k} [J_k(x)J_m(y) + 
J_m(x)J_k(y)] \int_{I^\delta_k} \int_{I^\delta_m}
\frac{dt ds}{|t-s|^\alpha} + \sum_{k=1}^N J_k(x)J_k(y)\int_{I^\delta_k} \int_{I^\delta_k}
\frac{dt ds}{|t-s|^\alpha} \Big). 
\end{equation*}
The integral of $|t-s|^{-\alpha}$ can be evaluated explicitly:
\begin{equation*}
\begin{aligned}
\frac{\kappa}{(1-\alpha)(2-\alpha)} 
\sum_{k=1}^N\sum_{m=1}^{k-1}
[J_k(x)J_m(y) + J_m(x)J_k(y)] \frac{h^2}{\delta^2} \Big( [(k-m)h + 
\delta]^{2-\alpha} + \\
 - 2[(k-m)h]^{2-\alpha}
+ [(k-m)h - \delta]^{2-\alpha}\Big)
+ \frac{\kappa}{(1-\alpha)(2-\alpha)} \sum_{k=1}^N 2J_k(x)J_k(y) 
\frac{h^2}{\delta^2} \delta^{2-\alpha}.
\end{aligned}
\end{equation*}
When $m < k$, the quantity between  parentheses
together with the $\delta^2$ on the denominator forms a centered difference 
approximation of the second order derivative of the function 
$r^{2-\alpha}$, evaluated at $(k-m)h$, i.e., at $t-s$. This 
derivative is precisely $(1-\alpha)(2-\alpha)|t-s|^{-\alpha}$. Meanwhile, the 
$h^2$ on the nominator can be viewed as the size of the measure $dt ds$ on 
each block $I_k \times I_m$.  Furthermore, $J_k(x)$ is precisely 
$L^h_{11}(x,t)$ evaluated on $I_k$.  The conclusion is: those terms in the 
above equation with $m < k$ form an approximation of
\begin{equation*}
\kappa \int_0^1 \int_0^1 \frac{L^h_{11}(x,t)L^h_{11}(y,s) }{|t-s|^\alpha} \ 
dt ds.
\end{equation*}
The second sum corresponds to the diagonal terms  $k=m$. Since $|J_k|$ is bounded, this sum is of order $O(h\delta^{-\alpha})$ and does not contribute in the limit as $h\to0$, as long as $\delta \gg h^{\frac 1 \alpha}$. $R^{h,\delta}_H$ converges to 
the covariance function of $\U_H(x;W^H)$, finishing the proof of 
\eqref{eq:clthl}.
\end{proof}

\section{A hybrid scheme that passes the corrector test}
\label{sec:disc}

We now present a method that eliminates the amplification effect of HMM with $\delta<h$
exhibited in item one of Theorem \ref{thm:hmm} when the random media has 
SRC. Such an effect arises because the fluctuation in the short-range 
averaging effects occurring on the interval of size $h$ are not properly 
captured by averaging occurring on an interval of size $\delta<h$. 

The main idea is to subdivide the element $I_k$ uniformly into $M$ smaller patches and perform $M$ independent calculations on each of these patches. This is a hybrid method that captures the idea of performing calculations on small intervals of size $\delta\ll h$ to reduce cost as in HMM while preserving the averaging property of MsFEM by solving the elliptic equation on the whole domain. 

 Let $\delta= h/M$ be the size of the  small patch $I^\ell_k$  for $1\leq \ell\leq M$. Define $b^k_{\eps \ell}$ by
\begin{equation}
b^k_{\eps \ell} = \left(\frac{\delta}{h}\right)^2 \left( \int_{I_k^\ell} 
a_\eps^{-1}(s) ds\right)^{-1}.
\label{eq:bkelldef}
\end{equation}
This definition is motivated by \eqref{eq:Aedex}. Given a function $w$ in the space $V^h_0$, we define its local 
projection into the space of oscillatory functions in the small patches 
$I_k^\ell$ by:
\begin{equation}
\left\{ \begin{aligned}
&\L_\eps (w^k_\ell) = 0, & & \quad x \in \cup_{k=1}^{N}\cup_{\ell=1}^M I_k^\ell,\\
& w^k_\ell = w, & & \quad x \in \cup_{k=1}^N \cup_{\ell=1}^M \partial 
I_k^\ell,
\end{aligned}
\right.
\label{eq:wkl}
\end{equation}
where $w^k_\ell$ denotes this local projection. Recall that
$\tilde{\phi}^k_0$ is the left piece of the hat basis function.
Integrations by parts show that $b^k_{\eps 
\ell}=A_\eps((\tilde{\phi}^k_0)_\ell, (\tilde{\phi}^k_0)_\ell)$,
where $A_\eps$ is the bilinear form defined in \eqref{eq:msfeml}. HMM 
choose one small patch $I^{\ell_*}_k$ and uses 
$A_\eps((\tilde{\phi}^k_0)_{\ell_*}, (\tilde{\phi}^k_0)_{\ell_*})=b^k_{\eps 
\ell_*}$ to approximate the value
$A_\eps(\tilde{\phi}^k_0, \tilde{\phi}^k_0)$. Of course, the scaling
$h/\delta$ is needed. This scaling factor turns out to amplify the variance 
as $h$ goes to zero when the random medium has SRC.

We modify the method of HMM by constructing $b_\eps$ as $b^k_\eps : = 
\textstyle \sum_{\ell=1}^M b^k_{\eps \ell}$. In other words, we define the 
stiffness matrix element $A^h_{\eps i-1 i}$ by $-\sum_{\ell=1}^M A_\eps( 
(\tilde{\phi}^i_0)_\ell, (\tilde{\phi}^i_0)_\ell )$.  With this definition, 
we verify that
\begin{equation*}
\begin{aligned}
b^k_\eps - b^k_0 &= \sum_{\ell = 1}^M \left(\frac{\delta}{h}\right)^2 
\Big[\Big(\int_{I_k^\ell} a_\eps^{-1} ds\Big)^{-1} - 
\Big(\int_{I_k^\ell} \astar^{-1} ds\Big)^{-1}\Big]\\
& = \sum_{\ell=1}^M \Big(\frac{\astar}{h}\Big)^2 \left[- \int_{I_k^\ell} 
q_\eps(s) ds + \Big(\int_{I_k^\ell} q_\eps(s) ds\Big)^2 
\Big(\int_{I_k^\ell} a_\eps^{-1}  ds\Big)^{-1}\right]
\end{aligned}
\end{equation*}
Rewriting the sum of the first terms in the parenthesis, we obtain
\begin{equation*}
b^k_\eps - b^k_0 = -\left(\frac{\astar}{h}\right)^2\int_{I_k} q_\eps(s) ds + 
r^k_\eps,
\end{equation*}
where $r^k_\eps$ accounts for the sum over the second terms in the 
parenthesis. Clearly, $\E|r^k_\eps| \le C\eps (h \delta)^{-1}$. This 
decomposition of $b_\eps - b_0$ and the estimate of $r^k_\eps$ shows that 
we can apply Proposition \ref{prop:msfemL} to obtain the decomposition of the 
corrector. The $L^h(x,t)$ function in this case will be $L^h_{11}(x,t)$ in 
\eqref{eq:Lsplit}. Then it follows from Propositions \ref{prop:wcone} and 
\ref{prop:wctwo} that the corrector in this method converges to the right 
limit.

In this modified method, all the local informations on $I_k$ 
are used to construct $b^k_\eps$ as in MsFEM. The main advantage is that the 
computation on $\{I_k^\ell\}_{\ell = 1}^M$ can be done in a parallel 
manner. The calculation in MsFEM performed on a whole domain of size $h$ is replaced by $h/\delta$ independent calculations. Accounting for the coupling between the $h/\delta$ subdomains is necessary in MsFEM. It is no longer necessary in the modified method, which significantly reduces is complexity.

\section{Conclusions and perspectives}
\label{sec:conclusion}

This paper analyzes the theory of random fluctuations for several multi-scale algorithms such as MsFEM and HMM in the simplified setting of a one-dimensional elliptic equation. One of our main results is that MsFEM and HMM with $\delta=h$ correctly capture the random fluctuations beyond the homogenization limit when the random media satisfy the short-range correlation (SRC) assumptions or the long-range correlation (LRC) assumptions considered in this paper. Such schemes pass the corrector test we have introduced and are therefore likely to be reliable in the analysis and the assessment of uncertainty quantifications even in settings where a well-understood theory of random fluctuations is not available. 

We have also shown that less expensive schemes such as HMM with $\delta<h$ still capture the random fluctuations in media with LRC but may fail to capture them in media with SRC. The reason is that media with LRC display self-similarities across scales. As a consequence, the properties of the random fluctuations are correctly captured at the macroscopic scale $h$. On the other hand, media with SRC display averaging effects at the scale $\eps\ll h$. Numerical schemes thus need to solve the equation at the microscopic scale on the whole domain in order to capture such effects. This forces the choice $\delta=h$ in the standard HMM scheme and led us to the modification of the MsFEM and HMM schemes presented in section \ref{sec:disc} to pass the corrector test for both LRC and SRC media.

The random fluctuations considered here all have a Gaussian structure and follow from an application of the central limit theorem for SRC media and a different but similar averaging procedure for LRC. Random fluctuations corresponding to rare events have to be analyzed with different methods. We refer the reader to \cite{BGL-JDE-11} for the one dimensional elliptic equation in the setting of large deviation theory.

The theory presented in this paper is restricted to the one-dimensional case. The main reason is that the theory of random fluctuations is well understood only in one dimension; see \cite{BGMP-AA-08,BP-99}. In this setting, the random elliptic solution can be written as a stochastic integral whose analysis follows from fairly standard techniques. Note that the discretized solution solves a discrete system whose inversion is not explicit. This explains in large part why the results presented in this paper are rather long and technical. 

In higher dimensions, only specific equations are amenable to fluctuations analyses; see \cite{B-CLH-08,FOP-SIAP-82}. For such equations, the homogenization limit is trivial. Unlike the case considered in this paper, purely deterministic discretizations and basis functions based on the unperturbed elliptic equation provide accurate numerical solutions. Such models thus do not capture the difficulties inherent to the simulation of elliptic equations. They do not allow us to analyze, for instance, the resonance effects observed for MsFEM schemes and the over-sampling techniques designed to eliminate such effects \cite{EHW-SINUM-00,HW-JCP-97,HWC-MC-99}. The theory of  fluctuations and its effect on multi-scale algorithms remains an open problem in such a setting.

Even though the one dimensional equation is of purely academic interest, the analyses presented in this paper quantify the influence of using less expensive numerical schemes on the accuracy of the numerical solutions. In the case of SRC, random fluctuations are captured only by fairly expensive schemes that solve the elliptic equation on the whole domain of interest. In the case of LRC, self-similarity properties allow us to use much coarser schemes without sacrificing accuracy. These guiding rules should prevail for a large class of equations, elliptic and non-elliptic, linear and nonlinear, and hopefully in practical settings where a well-controlled theory of random fluctuations is not available.

\section*{Acknowledgments} The authors acknowledge very fruitful discussions with Weinan E on multi-scale algorithms and thank the reviewers for a thorough reading of the manuscript and remarks that helped with the presentation of the results. This work was partially funded by NSF grant DMS-0804696.

\appendix
\section{Several useful lemmas}
\subsection{Two moments}

We first record two key estimates on oscillatory integrals, which we have 
used already. The first one accounts for random media with SRC.
\begin{lemma}
\label{lem:oint}
Let $q(x,\omega)$ be a mean-zero stationary random process with integrable 
correlation function $R(x)$. Let $[a, b]$ and $[c, d]$ be two intervals on 
$\R$ and assume $b -a \le d -c$.  Then
\begin{equation}
\left\lvert \E \int_a^b \int_c^d q_\eps(t) q_\eps(s) dt ds \right\rvert \le 
\eps (b-a)\|R\|_{1,\R}.
\end{equation}
\end{lemma}
\begin{proof}
Let $T$ denotes the expectation of the double integral. It has the 
following expression:
$$
T = \int_a^b\int_c^d R(\frac{t -s}{\eps}) dt ds =
 \int_{\R}\int_{\R} R(\frac{t-s}{\eps}){\bf 1}_{[a,b]}(t){\bf 1}_{[c, 
d]}(s) dt ds.
$$
We change variables by setting $t \to t$ and $(t-s)/\eps \to s$. The Jacobian 
of this change of variables is $\eps$. Then we have
$$
|T| \le \eps \int_{\R}\int_{\R} |R(s)| {\bf 1}_{[a,b]}(t) dt ds = 
\eps(b-a)\|R\|_{1, \R}.
$$
This completes the proof.
\end{proof}

The second one accounts for a special family of random media with LRC.
The proof is adapted from \cite{BGMP-AA-08}.
\begin{lemma}
\label{lem:ointl}
Let $q(x,\omega)$ be defined as in {\upshape (L1)-(L3)}.  Let $F$ be a 
function in the space $L^\infty(\R)$. Let $(a,b)$ and $(c,d)$ be two open 
intervals and assume $b-a \le d-c$. Then we have
\begin{equation}
\left| \E \int_a^b \int_c^d q({\frac t \eps})q({\frac s \eps}) F(t)F(s) dt 
ds \right| \le C \eps^{\alpha}(b-a)(d-c)^{1-\alpha}.
\label{eq:ointl}
\end{equation}
The constant $C$ above depends only on $\kappa$, $\alpha$ and 
$\|F\|_{\infty,\R}$.
\end{lemma}
\begin{proof} By the definition of the correlation function $R$, we have
$$
\E\Big\{\frac{1}{\eps^{\alpha}} \int_a^b \int_c^d q(\frac{t}{ 
\eps})q(\frac{s}{\eps}) F(t)F(s) dt ds \Big\} = \int_{\R^2} \eps^{-\alpha} 
R(\frac{t-s}{\eps}) F(t){\bf 1}_{[a,b]}(t) F(s){\bf 1}_{[c,d]}(s) dt ds.
$$
As shown in \cite{BGMP-AA-08}, $R(\tau)$ is asymptotically 
$\kappa\tau^{-\alpha}$ with $\kappa$ defined in \eqref{eq:kappa}.  We 
expect to replace $R$ by $\kappa \tau^{-\alpha}$ in the limit.  Therefore, 
let us consider the difference
\begin{equation*}
\int_{\R^2} \left| \eps^{-\alpha} R(\frac{t-s}{\eps}) - 
\frac{\kappa}{|t-s|^{\alpha}} \right| |F(t)|{\bf 1}_{[a,b]}(t) |F(s)|{\bf 
1}_{[c,d]}(s) dt ds.
\end{equation*}
By the asymptotic relation $R \sim \kappa \tau^{-\alpha}$, we have for any 
$\delta > 0 $, the existence of $T_\delta$ such that
\begin{math}
|R(\tau) - \kappa \tau^{-\alpha}| \le \delta \tau^{-\alpha}.
\end{math}
Accordingly, we decompose the domain of integration into three subdomains:
\begin{equation*}
\begin{aligned}
D_1 & = \{(t,s) \in \R^2, |t-s| \le T_\delta \eps\},\\
D_2 & = \{(t,s) \in \R^2, T_\delta \eps < |t-s| \le 1\},\\
D_3 & = \{(t,s) \in \R^2, 1 < |t-s|\}.
\end{aligned}
\end{equation*}
On the first domain, we have
\begin{equation*}
\begin{aligned}
\int_{D_1} & \left|  \eps^{-\alpha} R(\frac{t-s}{\eps}) - 
\frac{\kappa}{|t-s|^{\alpha}} \right| |F(t)|{\bf 1}_{[a,b]}(t) |F(s)|{\bf 
1}_{[c,d]}(s) dt ds\\
& \le \int_{D_1} \left| \eps^{-\alpha} R(\frac{t-s}{\eps})\right||F_1(t)| 
|F_2(s)| dt ds + \int_{D_1} \left| \frac{\kappa}{|t-s|^{\alpha}} \right| 
|F_1(t)||F_2(s)| dt ds.
\end{aligned}
\end{equation*}
Here and below, we use the short hand notation $F_1(t) = F(t){\bf 
1}_{[a,b]}(t)$ and $F_2(s) = F(s){\bf 1}_{[c,d]}(s)$. The above integrals 
are then bounded by
\begin{equation*}
\begin{aligned}
& \eps^{-\alpha}\|R\|_{\infty,\R} \int_a^b |F(t)| \int_{t-T_\delta \eps}^{t 
+ T_\delta \eps} |F_2(s)| ds dt + \int_a^b |F(t)| \int_{-T_\delta 
\eps}^{T_\delta \eps} \kappa |s|^{-\alpha} |F_2(t-s)| ds dt\\
\le & \|F\|_{\R,\infty}^2  \Big(2T_\delta \|R\|_{\R,\infty} + \frac{2\kappa 
T_\delta^{1-\alpha}}{1-\alpha} \Big) \eps^{1-\alpha}.
\end{aligned}
\end{equation*}

On domain $D_2$, we have
\begin{equation*}
\begin{aligned}
\int_{D_2} & \left|  \eps^{-\alpha} R(\frac{t-s}{\eps}) - 
\frac{\kappa}{|t-s|^{\alpha}} \right| |F_1(t)| |F_2(s)|dt ds
\le \delta \int_{D_2} |t-s|^{-\alpha} |F_1(t)| |F_2(s)| dt ds 
\\
& \le 2 \delta \int_a^b |F(t)| \int_{T_\delta \eps}^1 |s|^{-\alpha} 
|F_2(t-s)| ds dt \le \frac{2 \delta \|F\|_{\infty,\R}^2}{1-\alpha} 
(1+T_\delta^{1-\alpha}\eps^{1-\alpha}).
\end{aligned}
\end{equation*}

On domain $D_3$, we can bound $|t-s|^{-\alpha}$ by one, and we have
\begin{equation*}
\begin{aligned}
\int_{D_3} & \left|  \eps^{-\alpha} R(\frac{t-s}{\eps}) - 
\frac{\kappa}{|t-s|^{\alpha}} \right| |F_1(t)| |F_2(s)|dt ds
\le \delta \int_{D_3} |t-s|^{-\alpha} |F_1(t)| |F_2(s)| dt ds \\
& \le 2 \delta \int_a^b |F(t)| \int_{T_\delta \eps}^1 |F(t-s)| ds dt \le 
2\delta\|F\|_{\infty,\R}^2 (1+T_\delta \eps).
\end{aligned}
\end{equation*}
Therefore, for some constant $C$ that does not depend on $\eps$ or 
$\delta$, we have
\begin{equation*}
\limsup_{\eps \to 0} \eps^{-\alpha} \left| \E\int_a^b\int_c^d q({\frac t 
\eps})q({\frac s \eps}) F_1(t) F_2(s)- \int_{\R^2} 
\frac{\kappa}{|t-s|^{\alpha}}F_1(t) F_2(s) \right| \le C 
\|F\|_{\infty,\R}^2 \delta.
\end{equation*}
Sending $\delta$ to zero, we see that
\begin{equation}
\lim_{\eps \to 0} \eps^{-\alpha}  \E\int_a^b\int_c^d q({\frac t 
\eps})q({\frac s \eps}) F_1(t) F_2(s) dt ds =  \int_{\R^2} 
\frac{\kappa}{|t-s|^{\alpha}}F_1(t) F_2(s) dt ds.
\label{eq:lkey1}
\end{equation}

Finally, from the Hardy-Littlewood-Sobolev inequality \cite[\S 4.3]{LL}, we 
have
\begin{equation}
\left|\int_{\R^2} \frac{F_1(t)F_2(s)}{|t-s|^{\alpha}} dt ds \right| \le 
C\|F_1\|_{1,\R}\|F_2\|_{(1-\alpha)^{-1},\R} \le 
C\|F\|^{2}_{\infty}(b-a)(d-c)^{1-\alpha}.
\label{eq:lkey2}
\end{equation}
This completes the proof.
\end{proof}
\subsection{Moment bound for stochastic process}

In this section we provide a bound for the fourth order moment of 
$q(x,\omega)$ in terms of the $L^1$ norm of the $\rho$-mixing coefficient.  

Let $\mathcal{P}$ be the set of all ways of choosing pairs of points in $\{1,2,3,4\}$, i.e.,
\begin{equation}
\mathcal{P} : = \Big\{ p = \big\{ \{p(1), p(2)\}, \{p(3), p(4)\} \big\} 
~\Big|~ p(i) \in \{1,2,3,4\} \Big\}.
\label{eq:sPdef}
\end{equation}
There are $C_6^2 = 15$ elements in $\mathcal{P}$.

\begin{lemma}
\label{lem:fmt}
Let $q(x,\omega)$ be a stationary 
mean-zero stochastic process. Assume $\E |q(0)|^4$ is finite and
$q(x,\omega)$ is $\rho$-mixing with mixing coefficient $\rho(r)$ that is 
decreasing in $r$.  Then we have
\begin{equation}
|\E \{\prod_{i=1}^4 q(x_i) \}| \le \E|q(0)|^4 \sum_{p \in \mathcal{P}} 
\rho^{\frac 1 2} (|x_{p(1)} - x_{p(2)}|) \rho^{\frac 1 2}(|x_{p(3)} - 
x_{p(4)}|).
\label{eq:fmt}
\end{equation}
\end{lemma}
\begin{proof}
Given four points $\{q(x_i)\}$, $i = 1, \cdots, 4$, we can draw six line 
segments joining them. Among these line segment there is one that has the 
shortest length. Rearranging the indices if necessary, we assume it is the 
one joining $x_1$ and $x_2$. Then set $A = \{x_1, x_2\}$ and $B = \{x_3, 
x_4\}$. Rearranging the indices among each set if necessary, we assume also 
that $d(A, B)$ is obtained by $|x_1 - x_3|$. Then by the definition of 
$\rho$-mixing, we have
\begin{equation}
|\E\{\prod_{i = 1}^4 q(x_i)\} - R(x_1 - x_2)R(x_3 - x_4)| \le 
\mathrm{Var}\{ q(x_1)q(x_2)\}^{\frac 1 2} 
\mathrm{Var}\{q(x_3)q(x_4)\}^{\frac 1 2} \rho(|x_1 - x_3|).
\end{equation}
We can bound $\mathrm{Var}\{q(x_1)q(x_2)\}$, and similarly the variance of 
$\mathrm{Var}\{q(x_3)q(x_4)\}$, from above by $(\E|q(x_1)|^4 
\E|q(x_2)|^4)^{1/2}$.  Therefore, the expression above can be bounded by $\E |q(0)|^4 
\rho (|x_1 - x_3|)$.

Since $\rho$ is decreasing and $|x_1 - x_3| \ge |x_1 - x_2|$, we also have
\begin{equation}
|\E\{\prod_{i = 1}^4 q(x_i)\} - R(x_1 - x_2)R(x_3 - x_4)| \le \E|q(0)|^4 
\rho(|x_1 - x_2|).
\end{equation}
Now observe that $\min\{a,b\} \le (ab)^{\frac 1 2}$ for any two 
non-negative real numbers $a$ and $b$. Applying this observation to the bounds 
of the two inequalities above, and using the triangle inequality, we obtain 
\begin{equation}
|\E\{\prod_{i = 1}^4 q(x_i)\}| \le |R(x_1 - x_2)|\cdot|R(x_3 - x_4)| + \E|q(0)|^4 \rho^{\frac 1 2}(|x_1 - x_2|) \rho^{\frac 1 2}(|x_1 - x_3|).
\end{equation}
Using the definition of mixing 
again, we obtain that $|R(x_1 - x_2)| = |\E q(x_1)q(x_2)|$ is bounded by
$$
\mathrm{Var}^{\frac 1 2}(q(x_1)) 
\mathrm{Var}^{\frac 1 2}(q(x_2))\rho(|x_1 - x_2|) \le (\E|q(0)|^4)^{\frac 1 
2}\rho^{\frac 1 2}(|x_1 - x_2|).
$$
In the last step, we used the fact that $\rho \le \rho^{\frac 1 2}$ since 
$\rho$ can always be chosen no larger than 1. We can bound $R(x_1 - x_3)$ 
in the same way.  Therefore, we obtain
\begin{equation}
|\E\{\prod_{i = 1}^4 q(x_i)\}| \le  \E|q(0)|^4 \Big[ \rho^{\frac 1 2}(|x_1 
- x_2|) \rho^{\frac 1 2}(|x_3 - x_4|) + \rho^{\frac 1 2}(|x_1 - x_2|) 
\rho^{\frac 1 2}(|x_1 - x_3|)\Big].
\end{equation}
This completes the proof.
\end{proof}

We now  derive a bound for the fourth order moment of oscillatory 
integrals of $q_\eps$.
\begin{lemma}
\label{lem:fmtint}
Let $q(x,\omega)$ satisfy the conditions in the previous lemma. Assume in 
addition that the mixing coefficient satisfies that $\|\rho^{\frac 1 
2}\|_{1,\R_+}$ is finite. Let $(x,y)$ be an interval in $\R$. Then for any 
bounded function $m(t)$, we have
\begin{equation}
\E\left(\frac{1}{\sqrt{\eps}} \int_x^y q(\frac{t}{\eps}) m(t) dt \right)^4 
\le 60 \E|q(0)|^4 \cdot \|\rho^{\frac 1 2}\|_{1, \R_+}^2 \|m\|_{\infty}^4 
\cdot |x-y|^2.
\label{eq:fmtint}
\end{equation}
\end{lemma}
\begin{proof} The left hand side of the desired inequality is
\begin{equation}
I = \frac{1}{\eps^2} \int_x^y \int_x^y\int_x^y\int_x^y \E \prod_{i=1}^4 
q(\frac{t_i}{\eps}) \prod_{i=1}^4 m(t_i) d[t_1 t_2 t_3 t_4].
\end{equation}
Here and below, $d[t_1\cdot t_4]$ is a short-hand notation for $dt_1\cdots 
d_4$. Apply the preceding lemma, we have
$$
I \le \frac{\E |q(0)|^4 \|m\|_{\infty}^4 }{\eps^2} \sum_{p \in 
\mathcal{P}} \int_x^y\int_x^y\int_x^y\int_x^y \rho^{\frac 1 
2}(\frac{t_{p(1)} - t_{p(2)}}{\eps})\rho^{\frac 1 2}(\frac{t_{p(3)} - 
t_{p(4)}}{\eps}) d[t_{p(1)}\cdots t_{p(4)}].
$$
Note that we did not write absolute sign for the argument in the $\rho$ 
functions. We assume $\rho$ is extended to be defined on the whole $\R$ by 
letting $\rho(x) = \rho(|x|)$. There are $15$ terms in the sum above that are estimated in the same manners.  Let us look at one of them, with 
$p(1) = p(3) = 1$, $p(2) = 2$, and $p(4) = 3$.  We perform the following 
change of variables:
$$
\frac{t_1 - t_2}{\eps} \to t_2, \quad \frac{t_1 - t_3}{\eps} \to t_3, \quad  
t_1 \to t_1, \quad t_4 \to t_4.
$$
The Jacobian resulting from this change of variable cancels $\eps^2$ on the 
denominator. The integral becomes
\begin{equation}
\int_x^y dt_1 \int_x^y dt_4 \int_{\frac{t_1 - y}{\eps}}^{\frac{t_1 - 
x}{\eps}} \rho^{\frac 1 2}(t_2) dt_2\int_{\frac{t_1 - y}{\eps}}^{\frac{t_1 
- x}{\eps}} \rho^{\frac 1 2}(t_3) dt_3.
\end{equation}
This integral is finite and is bounded from above by
\begin{equation}
|x-y|^2 \|\rho^{\frac 1 2}\|_{1,\R}^2 = 4|x-y|^2\|\rho^{\frac 1 
2}\|_{1,\R_+}^2.
\end{equation}
The other terms in the sum have the same bound. Hence we have,
\begin{equation}
I \le \E|q(0)|^4 \times 15 \times 4|x-y|^2\|\rho^{\frac 1 2}\|_{1,\R_+}^2 
\|m\|_{\infty}^4.
\end{equation}
This verifies \eqref{eq:fmtint} and completes the proof.
\end{proof}

\subsection{Fractional Brownian motion}

For the convenience of the reader, we briefly review some essential 
properties of fractional Brownian motion (fBm), and stochastic integral 
with fBm integrator.

A fBm $W^H(t)$ with Hurst index $H$ is a mean-zero Gaussian process with 
$W^H(0) = 0$, stationary increments and $H$-self-similarity, that is, for 
$a > 0$,
\begin{equation}
\{ W^H(at)\}_{t \in \R} \overset{\mathscr{D}}{=} \{a^H W^H(t)\}_{ t \in 
\R},
\label{eq:Hss}
\end{equation}
where $\overset{\mathscr{D}}{=}$ means the equality in the sense of finite 
dimensional distributions. From this similarity relation, we deduce 
$\E[(W^H(t))^2] = |t|^{2H}\E[(W^H(1))^2]$. In particular, if 
$\E[(W^H(1))^2] = 1$ we say the fBm is standard. It follows from the 
stationarity of increments that the covariance function of $W^H(t)$ is 
given by
\begin{equation}
R^H(t,s) = \E\{W^H(t) W^H(s)\} = \frac{1}{2}\left(|t|^{2H} + |s|^{2H} - 
|s-t|^{2H} \right).
\label{eq:RHdef}
\end{equation}
When $H = 1/2$, the increments of $W^H$ are independent and the fBm reduces 
to the usual Brownian motion. For $H \ne 1/2$, the increments are 
stationary but not independent. 

Stochastic integrals with respect to fBm can be defined on many functional 
spaces. Note that $H = 1 - {\frac \alpha 2}$ is in the interval 
$(\frac{1}{2}, 1)$. In this case, a convenient functional space to define 
stochastic integral is
\begin{equation}
|\Gamma|^H = \left\{ f: \int_\R \int_\R |f(x)| |f(y)| |x - y|^{2(H-1)} dx 
dy < \infty \right\}.
\label{eq:sispace}
\end{equation}
It is easy to check, for instance from Hardy-Littlewood-Sobolev lemma 
\cite[\S 4.3]{LL}, that $L^1(\R)\cap L^2(\R) \subset L^{1/H} \subset 
|\Gamma|^H$.  Stochastic integrals against fBm do not satisfy It\^{o} 
isometry; instead, we have
\begin{equation}
\E\left\{\int_\R f(t) dW^H_t \int_\R h(s) dW^H_s \right\} = 
H(2H-1)\int_{\R^2} \frac{f(t) h(s)}{|t-s|^{2(1-H)}} dt ds.
\label{eq:lito}
\end{equation}

For a nice review on stochastic integral with respect to fractional 
Brownian motion, we refer the reader to \cite{VT-PTRF-00}.


\end{document}